\documentclass[reqno,a4paper]{amsart}

    \usepackage[utf8]{inputenc}
    \usepackage{fullpage}
    \usepackage{setspace}
    \usepackage{microtype}
    \usepackage[dvipsnames]{xcolor}
    \usepackage[colorlinks = true, linkbordercolor = {white}, urlcolor={purple}, linkcolor={purple}, citecolor = {purple}]{hyperref}
    \usepackage{cleveref}
    \usepackage{mathrsfs}
    \usepackage{amsmath}
    \usepackage{amsthm}
    \usepackage{amssymb}
    \usepackage{float}
    \usepackage{mathtools}
    \usepackage{tikz}
        \usetikzlibrary{patterns}
        \usetikzlibrary{graphs,graphs.standard,quotes}
        \usetikzlibrary{arrows}
    \usepackage[foot]{amsaddr}
    \usepackage{marginnote}
    \usepackage{enumitem}

    \newcommand{\R}{\mathbb{R}}
    \newcommand{\N}{\mathbb{N}}
    \newcommand{\Z}{\mathbb{Z}}
    \renewcommand{\emph}{\textsl}
    \renewcommand{\textit}{\textsl}

    \theoremstyle{definition}
        \newtheorem{thrm}{Theorem}[section]
        \newtheorem{lem}[thrm]{Lemma}
        \newtheorem{prop}[thrm]{Proposition}
        \newtheorem{cor}[thrm]{Corollary}
        
        \newtheorem{defin}[thrm]{Definition}
        
        \newtheorem{exmp}[thrm]{Example}
        \newtheorem{nonotheorem}{Theorem}
            
        \newtheorem{nonocorollary}[nonotheorem]{Corollary}
    \theoremstyle{remark}
        
    \theoremstyle{plain}
        \newtheorem{claimA}[]{Claim}
            
        \newtheorem{claim}[]{Claim}

    \Crefname{exmp}{Example}{Examples}
    \Crefname{claimA}{Claim}{Claims}

\title{The matrix potential game and structures of self-affine sets}

\date{\today}

\allowdisplaybreaks

\begin{document}

\author[R.\,A.\,Howat]{R.\,A.\,Howat}
\author[A.\,Mitchell]{A.\,Mitchell}
\author[T.\,Samuel]{T.\,Samuel }

\address[R.\,A.\,Howat]{School of Mathematics, University of Birmingham, Birmingham, UK}
\address[A.\,Mitchell]{Department of Mathematical Sciences, Loughborough University, Loughborough, UK}
\address[T.\,Samuel]{Department of Mathematics and Statistics, University of Exeter, Exeter, UK}



\maketitle

% {\footnotesize
% \textls[-5]{$^{\,1}$ School of Mathematics, University of Birmingham, Birmingham, UK. $^{\,2}$ Department of Mathematical Sciences, Loughborough\\
% \null\hspace{0.85em}University, Loughborough, UK. $^{\,3}$ Department of Mathematics and Statistics, University of Exeter, Exeter, UK.\\[0.5em]
% $^{\,*}$\hspace{0.2em}Corresponding author. a.samuel@exeter.ac.uk}
% }

\begin{abstract}
We present a new variant of the potential game and show that certain compact subsets of $\R^n$, including a large class of self-affine sets, are winning in our game. We prove that sets with sufficiently strong winning conditions are non-empty, provide a lower bound for their Hausdorff dimension, show that they have good intersection properties, and provide conditions under which, given $M \in \N$, they contain a homothetic copy of every set with at most $M$ elements. The applications of our game to self-affine sets are new and complement the work of Yavicoli \emph{et al.} (\emph{Int.\,Math.\,Res.\,Not.\,IMRN} 2023 and \emph{Math.\,Z.} 2022) for self-similar sets.
\end{abstract}

\section{Introduction}

Topological games, such as the \emph{Schmidt game} and its generalisations, are an important tool in metric number theory and fractal geometry.
Winning sets for the Schmidt game have full Hausdorff dimension and are countably stable under intersection.
Quantitative versions of the Schmidt game, such as the \emph{potential game} introduced by Broderick, Fishman and Simmons~\cite{broderick2017quantitative-OriginalRestrictedPotential} also have the property that countable intersections of winning sets are winning.
Since Hausdorff dimension is not stable under even finite intersections, topological games thus provide a useful tool for proving that sets have large intersection.

One of the main applications of
topological games has been in proving the existence of patterns in sets.
Hausdorff dimension alone is insufficient to show the existence of patterns; indeed, there exist examples of sets with full Hausdorff dimension that avoid patterns~\cite{MissingAllrectanglesKeleti,MissingcountablymeanytrianglessKeleti}.
On the other hand, for every $M \in \N$, there exists a set with non-full Hausdorff dimension that contains a homothetic copy of every set with at most $M$ elements~\cite{falconer2022intersections-ThicknessBasedGame,yavicoli2022thickness-BoundaryBasedGame}.
Much work has been dedicated to proving the existence of patterns in various families of sets, see for example~\cite{broderick2017quantitative-OriginalRestrictedPotential,pattsparcesets,falconer2022intersections-ThicknessBasedGame,Arithmetic_progressions_in_sets_of_fractional_dimension,Smallsetscontaining,OriginalAlexiaPaper,yavicoli2022thickness-BoundaryBasedGame}.
Notably, in~\cite{broderick2017quantitative-OriginalRestrictedPotential,falconer2022intersections-ThicknessBasedGame,OriginalAlexiaPaper,yavicoli2022thickness-BoundaryBasedGame}, the potential game was used to show the existence of patterns in several families of sets in $\R^n$.
These include a broad family of self-similar sets and compact sets with a certain cut-out structure.
However, a drawback of the potential game is that it is not applicable to attractors of strictly affine iterated function systems.

In the present paper, we introduce a new variant of the potential game, which we call the \emph{matrix potential game}.
A broad class of self-affine sets are winning for this game, and the techniques that we develop allow 
us to obtain results on the dimensions of such sets, with applications to patterns and intersections.
The game is played on $\R^n$, equipped with the square metric, and has defining parameters $\alpha,A,c,\rho_1,\rho_2$, where $\alpha,c\in(0,1)$, $\rho_1\geq\rho_2> 0$ and $A$ is a diagonal ($n\!\times\!n$)-matrix with diagonal entries $\beta_{11},\ldots,\beta_{nn} \in(0,1)$.
The variables $\rho_1$ and $\rho_2$ control the first play of the game for Player\,I, the matrix $A$ is used to construct the subsequent plays of Player\,I, and $\alpha$ restricts the possible responses by Player\,II.
Our main result provides sufficient conditions for when a winning set is non-empty and in many cases provides a lower bound on its Hausdorff dimension.

\begin{nonotheorem}\label{nnthm:winning-set-dimension}
    Let $A$ be a diagonal ($n\!\times\!n$)-matrix with diagonal entries $\beta_{11},\ldots,\beta_{nn}\in(0,1/5)$.  Let $\alpha, c \in (0,1)$ and $\rho_1, \rho_2 \in \mathbb{R}$ with $\rho_1\geq\rho_2 > 0$.
    If $S$ is winning in the $(\alpha,A,c,\rho_2,\rho_1)$-potential game, and satisfies an additional mild assumption, for all $y\in\R^n$ the set $S\cap \big(A(B[0,\rho_2])+y\big)$ is non-empty and
    \begin{align*}
        \dim_{H}\left(S\cap \left(A(B[0,\rho_2])+y\right)\right)\geq\max\left\{ n-K,0\right\},
    \end{align*}
    where $K$ is a constant dependent only on the parameters of the game. Moreover, there exist many examples for which $n-K > 0$.
\end{nonotheorem}

The precise statement of \Cref{nnthm:winning-set-dimension} is given in \Cref{thm:winning-set-dimension}. As a consequence, we obtain the following applications to patterns and intersections.

\begin{nonocorollary}\label{nncor:patterns}
    Assume the setting of Theorem~\ref{nnthm:winning-set-dimension}, and assume there exist $\gamma\in(0,1)$ and $M\in \N$ with
        \begin{align}\label{eq: equations in CorB}
            M\alpha^c\leq \gamma^{2}\left(1-\left(\prod_{j=1}^n\beta_{jj}\right)^{1-c}\right) \text{\,\,\,\,\, and \,\,\,\,\,} 
            \prod\limits_{j=1}^n\left( \frac{1}{3}(1-5\beta_{jj}^{\lfloor\gamma(M^{\frac{1}{c}}\alpha)^{-1}\rfloor})\right)>8^n (1+2^{2n+1})\gamma.
        \end{align} 
    For every winning set $S$ of the $(\alpha,A,c,\rho_2,\rho_1)$-potential game, every $y \in \R^n$, every finite set $C$ with at most $M$ elements and all $\lambda>0$ sufficiently small, there exists a non-empty set $X \subseteq \R^n$ such that for all $x \in X$, we have that $\lambda C + x \subseteq S \cap B[y,\rho_2]$.
\end{nonocorollary}

\begin{nonocorollary}\label{nncor:intersections}
    Assume the setting of \Cref{nncor:patterns}. 
    Let $F_1,\ldots,F_M$ be winning sets of the $(\alpha,A,c,\rho_2,\rho_1)$-potential game and let $y \in \R^n$.
    The set $(\cap_{i=1}^M F_i)\cap \big(A(B[0,\rho_2])+y\big)$ is non-empty and moreover
    \begin{align*}
    \dim_{H}((\cap_{i=1}^M F_i)\cap \big(A(B[0,\rho_2])+y\big))\geq \max\left\{ n-K,0\right\},
    \end{align*}
    where $K$ is a constant dependent only on the parameters of the game and $M$. Additionally, there exists a large class of sets for which $n-K > 0$.
\end{nonocorollary}

The precise statements of \Cref{nncor:patterns} and \Cref{nncor:intersections} are given in \Cref{thm:patterns} and \Cref{cor:intersections}, respectively. 
These results apply to a broad class of self-affine sets, as well as cut-out sets with rectangular cut-outs with sides parallel to the coordinate axes.
Consequently, we obtain the following.

\begin{nonocorollary}\label{nncor:self-affine-patterns}
    Given $M\in \N$, there exists a totally disconnected attractor of an iterated function system consisting of strictly affine maps that contains a homothetic copy of every set with at most $M$ elements.
    Moreover, such sets can be explicitly constructed.
\end{nonocorollary}

The existence of explicit examples of such sets is verified in \Cref{prop:self-affine-patterns}.
This extends the work of~\cite{yavicoli2022thickness-BoundaryBasedGame}, where the analogous result was established for self-similar sets.

\subsection*{Outline}
In \Cref{sec:preliminaries}, we introduce notation that we use throughout, and provide the necessary preliminaries on fractal geometry.
The definition of the matrix potential game is presented in \Cref{sec:mp-game}, where we also  prove several fundamental results about winning sets.
In \Cref{sec:winning-sets}, we provide a general strategy for proving sets are winning and present two broad families of winning sets, which are not captured by the classical potential game.
\Cref{sec:winng-sets-HD} contains the proof of our main result; namely, that winning sets with sufficiently strong winning conditions are non-empty and have non-zero Hausdorff dimension.
In \Cref{sec:applications}, we present applications of our results to patterns and intersections.
Additionally, we provide an interesting application to the distance sets of winning sets.

\section{Preliminaries}\label{sec:preliminaries}

Throughout this article, we work in the metric space $(\R^n,d_{\infty})$, where $d_{\infty}$ denotes the square metric, defined by $d_{\infty}(x,y) = \max_{i \in \{1,2, \dots, n\}} \lvert x_{i} - y_{i} \rvert$ for all $x=(x_{1},\dots,x_{n}), y=(y_{1},\dots,y_{n}) \in \R^{n}$.
Given $x \in \R^{n}$ and $r>0$, we let $B[x,r]$ denote the closed ball in $\R^{n}$ with centre $x$ and radius $r$ (with respect to $d_{\infty}$). 
Given $C\subseteq \R^n$, $\lambda\in\R$ and $b\in\R^n$ we write
    \begin{align*}
    \lambda C+b= \{x\in\R^n:x=\lambda x'+b \; \text{for some} \; x'\in C\}.
    \end{align*}
For a diagonal $(n\!\times\!n)$-matrix $A=\operatorname{diag}(a_{1,1},\ldots,a_{n,n})$ and $t\geq 0$, we write $A^t=\operatorname{diag} (a^t_{1,1},\ldots,a^t_{n,n})$.

We let $\N = \{1,2,3,\ldots\}$ denote the natural numbers and set $\N_0 = \N \cup \{0\}$. 
Given a finite subset $J$ of $\N$, we let $J^{\N}$ denote the space of all right-infinite words consisting of letters from the alphabet $J$, which we equip with the product topology inherited from the discrete topology on $J$. 
Given $k \in \N$, we let $J^k = \{j_1 \cdots j_k \, : \, j_i \in J\}$ denote the set of all finite words with letters in $J$ and write $\epsilon$ for the empty word.
For each $w \in J^{\N}$ and $i,j \in \N_0$ with $i \leq j$, we write $w_{[i,j]} = w_i \cdots w_j$.

Given $F \subseteq R^n$ and $s \geq 0$, the \emph{$s$-dimensional Hausdorff measure} of $F$ is the quantity
\begin{align*}
    \mathcal{H}^s(F)= \lim_{\delta\rightarrow0}\inf\left\{ 
 \sum\operatorname{diam}(U_i)^s:\{U_i\}_{i\in I} \; \text{is a countable $\delta$-cover of} \; F\right\}.
\end{align*}
The \emph{Hausdorff dimension} of $F$ is then defined by 
\begin{align*}
    \dim_H(F)= \inf\{s\geq 0:\mathcal{H}^s(F)=0\}=\sup\{s\geq 0:\mathcal{H}^s(F)=\infty\}.
\end{align*}

Given a finite non-empty alphabet $I$, an \emph{iterated function system (IFS)} is a collection $\Phi=\{\phi_i\}_{i\in I}$ of contracting self-maps on compact subsets of $\R^n$.
By a classical result of Hutchinson \cite{HutchinsonResult}, there exists a unique non-empty compact set $F$, called the \emph{attractor} of the IFS, satisfying
\begin{align*}
    F=\bigcup_{i\in I}\phi_i(F).
\end{align*}
If the IFS consists of similarities, then $F$ is called \emph{self-similar}.
More generally, if the IFS consists of affine maps, then $F$ is called \emph{self-affine}. 

Another broad class of sets that often have non-integer Hausdorff dimension is given by Moran sets. 
Let $\{T_{\omega}:\omega\in\Lambda_k\}_{k\in\N_0}$ be a collection of non-empty compact subsets of $\mathbb{R}^{n}$ such that 
\begin{enumerate}
    \item $\Lambda_0=\epsilon$, and for $k\in\N$, $\Lambda_k\subset \Lambda_{k-1}\times \N$ and $\Lambda_k$ is finite for all $k\in\N$, and
    \item $T_{\epsilon}$ is a closed ball or $\R^n$ and for each $k\in\N$ and $\omega=\omega_1\dots\omega_k\in\Lambda_k$ we have  $T_{\omega}$ is a closed ball and $T_{\omega}\subseteq T_{\omega_{1}\dots\omega_{k-1}}$.
\end{enumerate}
The \emph{Moran set} associated to $\{T_{\omega}:\omega\in\Lambda_k\}_{k\in\N_0}$ is defined by
    \begin{align*}
    E=\bigcap_{k\in\N_0}\bigcup_{\omega\in\Lambda_k}T_{\omega}.
    \end{align*}
For a detailed introduction to attractors of iterated function systems and Moran sets, see \cite{FractalGeometry,TechniquesInFractalGeometry,HutchinsonResult}.

\section{The matrix potential game}\label{sec:mp-game}

The potential game was introduced in \cite{MR3826896} and a quantitative variant was defined in \cite{broderick2017quantitative-OriginalRestrictedPotential}.
The potential game is a two player topological game, played in a complete metric space.  However, in what follows, to highlight the core ideas, we restrict our treatment to the $n$-dimensional Euclidean setting.

\subsection{The potential game}
In the simplest variant of the potential game, both players are given parameters $\alpha,\beta,c>0$ and the $(\alpha,\beta,c)$-potential game is played as follows:
On the first turn Player\;I plays some closed ball $B[x_1,r_1]$ and Player\;II responds by choosing (effectively deleting) a countable collection of closed balls $\mathcal{A}(B[x_1,r_1])$ satisfying:
    \begin{align*}
    \sum_{B\in \mathcal{A}(B[x_1,r_1])} \operatorname{rad}(B)^c\leq (\alpha r_1)^c,
    \end{align*}
where $\operatorname{rad}(B)$ denotes the radius of $B$. Note, the sets $B\in \mathcal{A}(B[x_1,r_1])$ need not be contained in $B[x_1,r_1]$. For $m\in\N\setminus\{1\}$, on the $m$-th turn of Player\,I, they play $B[x_m,r_m]\subseteq B[x_{m-1},r_{m-1}]$ satisfying the condition that $r_m\geq \beta r_{m-1}$ and Player\;II responds by choosing (effectively deleting) a countable collection of closed balls $\mathcal{A}(B[x_{m},r_{m}], \dots, B[x_{1},r_{1}])$ satisfying the condition that
    \begin{align*}
    \sum_{B\in \mathcal{A}(B[x_{m},r_{m}], \dots, B[x_{1},r_{1}])} \operatorname{rad}(B)^c\leq (\alpha r_m)^c.
    \end{align*}
Player\;I must also satisfy the general rule that $r_m \rightarrow 0$ as $m \rightarrow \infty$.

By construction there exists a single point $x_{\infty}$, with $\{x_{\infty}\}=\cap_{m\in\N}B[x_m,r_m]$, called the \emph{outcome} of the game. 
A set $S\subseteq \R^n$ is an $(\alpha,\beta,c)$-winning set, 
if Player\;II has a strategy ensuring that $x_{\infty}\in S$ or 
    \begin{align*}
    x_{\infty} \not\in \bigcup_{m\in\N} \mathcal{A}(B[x_{m},r_{m}], \dots, B[x_{1},r_{1}]).
    \end{align*}    
The variation most commonly used allows Player\;II to, instead of choosing a countable collection of closed balls, choose a countable collection of closed neighbourhoods of given closed sets $L$ from a predetermined collection $\mathcal{H}$. The condition on the radii is then replaced by an analogous condition on the size of the neighbourhood of the chosen sets $L\in \mathcal{H}$, see \cite{broderick2017quantitative-OriginalRestrictedPotential,MR3826896,yavicoli2022thickness-BoundaryBasedGame}.

In \cite{falconer2022intersections-ThicknessBasedGame,OriginalAlexiaPaper,yavicoli2022thickness-BoundaryBasedGame}, variants of Newhouse thickness were introduced, and it was shown that certain sets with positive thickness are winning in the potential game (with different collections $\mathcal{H}$ depending on the thickness variant).
However, these strategies are not applicable to attractors of (strictly) affine iterated function systems, since their thickness is often zero.

Here, we introduce a new variant of the potential game in $\R^n$.
In this variant, Player\;I's plays are contracting affine transformations of some closed ball, where the transformation matrix $A$ is a diagonal matrix with positive diagonal entries.
Player\;II then deletes a countable collection of sets of the form $A^k(B[0,r])+z$ for some $k \geq 1$ and $z\in \R^{n}$, where $B[0,r]$ is the first play of Player\,I.
An advantage of this game is that the winning strategies that we present later in \Cref{sec:winning-sets} do not require positive thickness.
Thus, it can be applied to self-affine sets.

\subsection{A new potential game framework}

Our variant of the potential game is similar to the game introduced by Broderick, Fishman and Simmons \cite{broderick2017quantitative-OriginalRestrictedPotential}. Namely, if $A$ is a similarity, then the matrix potential game and that of \cite{broderick2017quantitative-OriginalRestrictedPotential} coincide.
The introduction of a matrix into this game is very natural and has been done for Schmidt's original game by Kleinbock and Weiss \cite{MR2581371}, but not as far as the authors are aware for the potential game.

\begin{defin}\label{def:winning}
    Let $A$ be a diagonal $(n\!\times\!n)$-matrix with diagonal entries $\beta_{11},\dots, \beta_{nn}\in (0,1)$.
    Given $\rho_1\geq \rho_2>0$, $\alpha>0$ and $c\in[ 0,1)$, the $(\alpha,A,c,\rho_2,\rho_1)$-potential game is played with the following rules, with Player\;II having the option to skip their turn:
    \begin{itemize}
        \item On their first turn, Player\;I plays a closed set $U_1=A(B[0,r])+b_1$ satisfying $\rho_1\geq r\geq \rho_2$.
        \item On their first turn, Player\;II responds by selecting a collection of tuples
        \begin{align*}
            \mathcal{A}(U_1) = \left\{(q_{i,1},y_{i,1}) :  q_{i,1} \in \R \; \text{with} \; q_{i,1} \geq 1, y_{i,1}\in\R^n \; \text{and} \; i\in I_1\right\},
        \end{align*}
        where $I_1$ is an at most countable index set.
        This defines the deleted set 
        \begin{align*}
        \bigcup_{\substack{(q_{i,1},y_{i,1})\\\in\mathcal{A}(U_1)}} \left(A^{q_{i,1}}(B[0,r])+y_{i,1}\right),
        \end{align*}
        which will be used to define the winning conditions at the end of the game. If $c>0$, then Player\;II's collection must satisfy
            \begin{align*}
            \sum\limits_{i\in I_1}\prod_{j=1}^n\left( \beta_{jj}^{q_{i,1}}\right)^c\leq \left(\alpha\prod_{j=1}^n\beta_{jj}\right)^{c}\hspace{-0.625em}.
            \end{align*}
        If $c=0$, then Player\;II can only choose one pair $(q_{1,1},y_{1,1})$ which satisfies:
            \begin{align*}
            \prod_{j=1}^n\beta_{jj}^{q_{1,1}}\leq\alpha \prod_{j=1}^n\beta_{jj}.
            \end{align*}
        \item For each $m\in\N\setminus\{1\}$, on their $m$-th turn, Player\;I plays a closed set $U_m=A^{m}\left(B[0,r]\right)+b_m$ such that $U_m\subseteq U_{m-1}$.
        \item On their $m$-th turn, Player\;II responds by selecting a collection of tuples 
        \begin{align*}
            \mathcal{A}(U_m,\dots,U_1) = \left\{(q_{i,m},y_{i,m}): q_{i,m} \in \R \; \text{with} \; q_{i,m} \geq 1,  y_{i,m}\in\R^n \; \text{and} \; i\in I_m\right\},
        \end{align*}
        where $I_m$ is an at most countable index set. 
        This defines the deleted set
        \begin{align*}
            \bigcup_{\substack{(q_{i,m},y_{i,m})\\\in\mathcal{A}(U_m,\dots,U_1)}} \left(A^{q_{i,m}}(B[0,r])+y_{i,m}\right).
        \end{align*}
        If $c>0$, then Player\;II's collection must satisfy
        \begin{align*}
        \sum\limits_{i\in I_m}\prod_{j=1}^n\left( \beta_{jj}^{q_{i,m}}\right)^c\leq \left(\alpha\prod_{j=1}^n\beta_{jj}^{m}\right)^{c}\hspace{-0.625em}. 
        \end{align*}
        If $c=0$, then Player\;II can only choose one pair $(q_{1,m},y_{1,m})$ which satisfies:
            \begin{align*}
            \prod_{j=1}^n\beta_{jj}^{q_{1,m}}\leq\alpha \prod_{j=1}^n\beta_{jj}^{m}.
            \end{align*}
    \end{itemize}
\end{defin}
By construction, there exists a single point $x_{\infty}$, with $\{x_{\infty}\} =\cap_{m\in\N}U_m$, called the outcome of the game and a deleted region,
    \begin{align}\label{eq:winning-condition-deletion}
        \operatorname{Del}(\mathcal{A}(U_1),\mathcal{A}(U_1,U_2),\dots) = \bigcup_{m\in\N} \bigcup_{\substack{(q_{i,m},y_{i,m})\\\in\mathcal{A}(U_m,\dots,U_1)}} \left(A^{q_{i,m}}(B[0,r])+y_{i,m}\right).
    \end{align}
We say that a set $S\subseteq \R^n$ is an \textit{$(\alpha,A,c,\rho_2,\rho_1)$-winning set}, if Player\;II has a strategy ensuring that if $x_{\infty}\not\in \operatorname{Del}(\mathcal{A}(U_1),\mathcal{A}(U_1,U_2),\dots)$, then $x_{\infty}\in S$.
    
    The restriction on Player\,II's plays when $c=0$ allows us to create a strategy where only one set is deleted each turn and the notation is chosen for this to be consistent with a desired monotonicity property of the game, see \Cref{lem:monotonicity}.

    A skip by Player\;II can be considered as Player\;II playing a collection of tuples that satisfies the conditions of the game but which has no impact on the outcome of the game, for instance, such that $\mathcal{A}(U_m,\dots,U_1) \cap U_{m} = \emptyset$. Thus, skips can be ignored. 
    
    One can also define the \emph{strong} $(\alpha,A,c,\rho_2,\rho_1)$-potential game by requiring Player\;I's plays to satisfy $U_{m}=A^{t_m}(B[0,r])+b_m$ where $t_{1}=1$, $t_m\in(t_{m-1},t_{m-1}+1]$ for $m\geq 2$, and $(t_{m})_{m\in \N}$ diverges to infinity.
    Player\;II's conditions are then changed by replacing $\beta_{jj}^m$ with $\beta_{jj}^{t_m}$ for each $j \in \{ 1, 2, \dots, n\}$.
    The strong game is what is typically called the potential game, see for example \cite{broderick2017quantitative-OriginalRestrictedPotential,falconer2022intersections-ThicknessBasedGame,MR3826896,yavicoli2022thickness-BoundaryBasedGame}. 
    However, here, we do not consider it as it leads to weaker results.
    We note, however, that all results in Sections\;\ref{sec:mp-game}, \ref{sec:winng-sets-HD} and \ref{sec:applications} hold for the strong variant since strong winning implies regular winning.

\subsection{Properties of winning sets}

In this section we show that winning sets for the matrix potential game satisfy two important properties that one expects winning sets for topological games to satisfy.
Namely, monotonicity and closure under countable intersections.

\begin{lem}[Monotonicity]\label{lem:monotonicity}
    Let $A$ be a diagonal $(n\!\times\!n)$-matrix with diagonal entries $\beta_{11},\dots, \beta_{nn}\in(0,1)$, $\rho_1 \geq \rho_2 > 0$, $\alpha \in (0,1]$ and $c\in[ 0,1)$. If $S$ is an $(\alpha,A,c,\rho_2,\rho_1)$-winning set, then for $\alpha',c',\rho_2',\rho_1'$ satisfying $\alpha\leq \alpha'$, $c'\geq c$, $\rho_2\leq\rho_2'\leq\rho_1'\leq \rho_1$, the set $S$ is $(\alpha',A,c',\rho_2',\rho_1')$-winning.
\end{lem}

\begin{proof}
    We show there exists a winning strategy for Player\;II in the $(\alpha',A,c',\rho_2',\rho_1')$-potential game.
    To do so, we show that for every valid play by Player\;I, Player\;II can respond using their winning strategy for the $(\alpha,A,c,\rho_2,\rho_1)$-potential game.

    Given $m\in\N$ and $U_m,\ldots, U_1$ played by Player\;I, let
        \begin{align*}
        \mathcal{A}(U_m,\dots,U_1)=  \left\{(q_{i,m},y_{i,m}):q_{i,m}\geq 1 \; \text{and} \; y_{i,m}\in\R^n\right\}_{i\in I_m}
        \end{align*}
    be Player\;II's $m$-th turn in the strategy for $S$ in the $(\alpha,A,c,\rho_2,\rho_1)$-potential game.
    We claim that in the $(\alpha',A,c',\rho_2',\rho_1')$-potential game, by selecting the collection $\mathcal{A}(U_m,\dots,U_1)$ on the $m$-th turn, Player\;II can ensure that if $x_{\infty}\not\in \operatorname{Del}(\mathcal{A}(U_1),\mathcal{A}(U_1,U_2),\dots)$, then $x_{\infty}\in S$, and hence, $S$ is $(\alpha',A,c',\rho_2',\rho_1')$-winning.
    To this end, observe that, given $m\in\N$, the collection of tuples $\mathcal{A}(U_m,\dots,U_1)$ satisfies the required conditions, limiting the size of deletion, since if $c=c'=0$ the required inequalities are immediately satisfied; and if $0 =c < c'$ we have $\mathcal{A}(U_m,..,U_1)$ consists of only one set and therefore 
        \begin{align*}
        \sum_{i\in I_m} \prod\limits_{j=1}^n \left(\beta_{jj}^{q_{i,m}}\right)^{c'}
        =\sum_{i\in I_m}\left( \prod\limits_{j=1}^n \beta_{jj}^{q_{i,m}}\right)^{c'}
        =\left( \prod\limits_{j=1}^n \beta_{jj}^{q_{i,m}}\right)^{c'}
        \leq \left(\alpha\prod_{j=1}^n\beta_{jj}^{m}\right)^{c'}
        \leq \left(\alpha'\prod_{j=1}^n\beta_{jj}^{m}\right)^{c'}\hspace{-0.625em}.
        \end{align*}
    Thus, it remains to consider the case $c > 0$. Note, if $0< c\leq c' < 1$, then the following inequality holds for a collection $(x_i)_{i\in I}$, for some finite or countable index $I$ and where $x_{i} \in (0,1)$ for $i \in I$:
        \begin{align*}
        \left(\sum_{i \in I} x_i^{c'}\right)^{\frac{1}{c'}}\leq \left(\sum_{i \in I}x_i^{c}\right)^{\frac{1}{c}}\hspace{-0.625em}.
        \end{align*}
    Therefore,
    \begin{align*}
    \sum\limits_{i\in I_m}\left( \prod\limits_{j=1}^n \beta_{jj}^{q_{i,m}}\right)^{c'} 
    \leq \left(\sum\limits_{i\in I_m}\left( \prod\limits_{j=1}^n \beta_{jj}^{q_{i,m}}\right)^c\right)^{\frac{c'}{c}}
    \leq\left(\left(\alpha\prod_{j=1}^n\beta_{jj}^{m}\right)^c\right)^{\frac{c'}{c}}
    \leq \left(\alpha'\prod_{j=1}^n\beta_{jj}^{m}\right)^{c'}\hspace{-0.625em}.
    \end{align*}

    Note that $U_1=A(B[0,r])+b_1$ for some $b_1\in\R^n$ and $r>0$ by the rules of the game, and
    observe that $r\in[\rho_2',\rho_1']\subseteq [\rho_2,\rho_1]$, hence $U_1$ is a valid play by Player\;I in the $(\alpha,A,c,\rho_2,\rho_1)$-potential game and therefore, since the selection condition is satisfied, Player\;II can respond with $\mathcal{A}(U_1)$ using the winning strategy in the $(\alpha,A,c,\rho_2,\rho_1)$-potential game.

    By induction, given some play $U_m$ by Player\;I, Player\;II can legally respond using their winning strategy in the $(\alpha,A,c,\rho_2,\rho_1)$-potential game, since any $U_m$ can be viewed as a play in this game, and hence, $S$ is $(\alpha',A,c',\rho_2',\rho_1')$-winning.
\end{proof}

\begin{lem}[Countable intersections]\label{lem:countableintersections}
    Let $A$ be a diagonal $(n\!\times\!n)$-matrix with diagonal entries $\beta_{11},\dots, \beta_{nn}\in(0,1)$.
    Let $\{S_j\}_{j\in J}$ be an at most countable collection of sets such that for each $j\in J$, the set $S_j$ is a $(\alpha_j,A,c,\rho_2,\rho_1)$-winning set with $c>0$. If $\alpha = (\sum_{j\in J}(\alpha_j)^c)^{1/c}$ is finite, then $S= \cap_{j\in J} S_j$ is an $(\alpha,A,c,\rho_2,\rho_1)$-winning set.
\end{lem}

\begin{proof} 
    Given $m\in\N$ and plays $U_m,\dots, U_1$ by Player\;I, for each $j\in J$ let
        \begin{align*}
        \mathcal{A}_{m,j}(U_m,\dots,U_1)= \left\{(q_{i,m,j},y_{i,m,j}): i\in I_{m,j}, q_{i,m,j}\geq 1 \; \text{and} \; y_{i,m,j}\in\R^n \right\}
        \end{align*}     
    denote the collection selected by Player\;II on the $m$-th turn in the strategy for $S_j$.
    For each $m\in\N$, set 
        \begin{align*}
        \mathcal{A}_m(U_m,\dots,U_1) = \bigcup_{j \in J} \mathcal{A}_{m,j}(U_m,\dots,U_1).
        \end{align*}
    We show, in the $(\alpha,A,c,\rho_2,\rho_1)$-potential game, by playing $\mathcal{A}_m(U_m,\dots,U_1)$ on their $m$-th turn, Player\;II can ensure, if $x_{\infty}\not\in \operatorname{Del}(\mathcal{A}(U_1),\mathcal{A}(U_1,U_2),\dots)$, then $x_{\infty}\in S$, namely $S$ is $(\alpha,A,c,\rho_2,\rho_1)$-winning.
    To this end, fix $m \in \N$, and observe that
    \begin{align*}
        \sum_{i\in I_m}\left(\prod_{k=1}^n\beta_{kk}^{q_{i,m}}\right)^c\leq \sum_{j\in J}\left( \sum_{i\in I_{m,j}}\left(\prod_{k=1}^n\beta_{kk}^{q_{i,m,j}}\right)^c \right)\leq \sum_{j\in J}\left(\alpha_j\prod_{k=1}^n\beta_{kk}^{m}\right)^c=\left(\alpha\prod_{k=1}^n\beta_{kk}^{m}\right)^c.
    \end{align*}
    Hence, each defined collection $\mathcal{A}(U_m,\dots,U_1)$ is a valid play for Player\;II.
    
    Assume that $x_{\infty}\not\in \operatorname{Del}(\mathcal{A}_1(U_1),\mathcal{A}_2(U_1,U_2),\dots)$ and hence $x_{\infty}\not\in\operatorname{Del}(\mathcal{A}_{1,j}(U_1),\mathcal{A}_{2,j}(U_1,U_2),\dots)$ for each $j\in J$.
    Note, if $x_{\infty}\not\in S$, then $x_{\infty}\not\in S_j$ for at least one $j\in J$, this would be a contradiction since Player\;II is playing according to their strategy.
    Hence $x_{\infty}\in S$ and therefore $S$ is $(\alpha,A,c,\rho_2,\rho_1)$-winning.
\end{proof}

This result does not automatically mean the intersection is non-empty, as for sufficiently large $\alpha$, one could make the empty set winning.
In \Cref{sec:winng-sets-HD}, we give conditions under which winning sets are non-empty, and in such cases, obtain a lower bound for their Hausdorff dimension.

\section{Winning sets of the matrix potential game}\label{sec:winning-sets}

In this section we present a general strategy for proving that a set is winning.
We then use this strategy to show that several broad classes of sets are winning. 
These include examples that do not fall into the scope of winning sets for previously considered topological and potential games.
Notably, several examples of self-affine sets fall into the families of sets we consider.

\subsection{A general strategy for closed sets}

Let $n \in \N$, $r > 0$, $F \subseteq B[0,r] \subset \R^{n}$ be a non-empty closed set, and $A$ be a diagonal $(n\!\times\!n)$-matrix with diagonal entries $\beta_{11},...,\beta_{nn} \in (0,1)$. Given a bounded sequence $(a_{n})_{n\in\N}$ of positive real numbers, $\rho_1\geq r$ and $c\in(0,1)$, we say that $\mathcal{F}= \{\mathcal{Q}_1,\mathcal{Q}_2,\dots\}$, with $\mathcal{Q}_{k}=\{(q_{i,k},y_{i,k}) : i\in I_k, \;q_{i,k}\geq 1 \; \text{and} \; y_{i,k}\in\R^n\}$, where $I_{k}$ is at most countable, is a strategy defining iterative covering (SDIC) of $B[0,r]\setminus F$, with respect to the tuple $(A, (a_{n})_{n\in \N}, \rho_{1}, c)$, if the following conditions are satisfied.
\begin{enumerate}[leftmargin=4.75em]
    \item[(SDIC~1)] The set $B[0,r]\setminus F$ is contained in
        \begin{align*}
        \bigcup_{k\in\N} \bigcup_{\substack{(q,y)\in\mathcal{Q}_{k}}}\left( A^{q}(B[0,r])+y \right).
        \end{align*}
    \item[(SDIC~2)]
    For all $z\in\R^n$ and $k\in \N$,
    \begin{align*}
        \sum_{\substack{(q,y)\in \mathcal{Q}_{k} \; \text{with}\\ (A^q(B[0,r])+y)\cap (A^k(B[0,\rho_1])+z)\neq \emptyset}} \left(\prod_{j=1}^n\beta_{jj}^q\right)^c\leq \left(a_k\prod_{j=1}^n \beta_{jj}^{k}\right)^c\hspace{-0.625em}.
    \end{align*}
\end{enumerate}

\begin{prop}\label{cutoutwin}
    Let $n \in \N$, let $r > 0$, and let $A$ be an $(n\!\times\!n)$-diagonal matrix with diagonal entries $\beta_{11},\dots,\beta_{nn}\in (0,1)$.
    If $F\subseteq B[0,r]$ is a non-empty closed set and $\mathcal{F}$ is a strategy defining iterative covering of $B[0,r]\setminus F$ with respect to the tuple $(A, (a_{n})_{n\in \N}, \rho_{1}, c)$, then $F\cup (\R^n\setminus B[0,r])$ is $(\alpha,A,c,r,\rho_1)$-winning, where $\alpha = \sup_{k\in\N}\{a_k\}$.
\end{prop}

\begin{proof}
    Let $m\in\N$ and let $U_m=A^m(B[0,l])+b_m$ denote the $m$-th play by Player\;I, where $l\in[r,\rho_1]$. Suppose that Player\;II responds with the collection
    \begin{align}\label{eq:Pro_4_1_eq_1}
        \mathcal{A}(U_m,\dots,U_1)= \{(q,y)\in Q_m:(A^q(B[0,r])+y)\cap U_m\neq \emptyset\},
    \end{align}
    if it is legal to do so. Otherwise Player\;II skips their turn.
    
    We now verify that this is in fact a winning strategy.
    Let $U_1,U_2,\dots$ denote the sequence of plays by Player\;I, where Player\;II responds according to the above strategy.
    This game has outcome $x_{\infty}$.
    Assume $x_{\infty}\not\in \operatorname{Del}(\mathcal{A}(U_1),\mathcal{A}(U_1,U_2),\dots)$ and suppose for a contradiction that $x_{\infty}\not\in F\cup (\R^n\setminus B[0,r])$.
    In which case, $x_{\infty}\in B[0,r]\setminus F$ and hence by (SDIC~1), we have that
        \begin{align*}
        x_{\infty}\in \bigcup_{k\in\N} \bigcup_{\substack{(q,y)\in\mathcal{Q}_{k}}}\left( A^{q}(B[0,r])+y \right).
        \end{align*}
    Therefore, there exists at least one $k^*\in\N$ and at least one $(p,w) \in\mathcal{Q}_{k^*}$ such that $x_{\infty}\in A^p(B[0,r])+w\subseteq A^p(B[0,l])+w$.
    Since, by definition, $x_{\infty}\in U_{k^*}=A^{k^*}(B[0,l])+b_{k^*}$, this implies that the tuple
    $(p,w)$ belongs to the set $\{(q,y)\in Q_{k^*}:(A^q(B[0,r])+y)\cap U_{k^*}\neq \emptyset\}.$
    By \eqref{eq:Pro_4_1_eq_1} and (SDIC~2), we have
    \begin{align*}
    \sum_{\substack{(q,y)\in \\\mathcal{A}(U_{k^*},\dots,U_1)}}\prod_{j=1}^n(\beta^{q}_{jj})^c
    =\sum_{\substack{(q,y)\in \mathcal{Q}_{k^*} \; \text{with}\\ (A^q(B[0,r])+y)\cap U_{k^*}\neq \emptyset}}\prod_{j=1}^n(\beta^{q}_{jj})^c
    \leq \left(a_{k^*}\prod_{j=1}^n\beta_{jj}^{k^*}\right)^c\leq \left(\alpha\prod_{j=1}^n\beta_{jj}^{k^*}\right)^c\hspace{-0.625em}.
    \end{align*}
    Therefore, $\mathcal{A}(U_{m^*},\dots,U_1)$ is a legal play of the game by Player\;II and, hence $x_{\infty}$ belongs to the deleted region $\operatorname{Del}(\mathcal{A}(U_1),\mathcal{A}(U_1,U_2),\dots)$, yielding the desired contradiction.
\end{proof}

\subsection{Examples of sets satisfying \texorpdfstring{\Cref{cutoutwin}}{Proposition~4.1}}

We now present some families of sets that satisfy the conditions of \Cref{cutoutwin}.
Given $U,V \in \N \setminus \{1\}$, we let $A_{U,V}$ denote the matrix
    \begin{align}\label{eq:AUV}
    A_{U,V}=\begin{pmatrix}
    U^{-1} & 0 \\
    0 & V^{-1}\end{pmatrix}.
    \end{align}
The first family of examples we consider is a class of cut-out sets.
Many sets in this family, such as the self-affine Sierpinski carpets, have zero thickness, therefore the techniques developed in \cite{falconer2022intersections-ThicknessBasedGame} may not be applied in their given form.

\begin{exmp}\label{exa:RCO}
    Given $U,V\in\N\setminus\{1\}$, $m \in \N$ and $t>0$, we let $\mathcal{RCO}(U,V,m,t)$ denote the family of subsets of $\R^2$ that can be constructed as follows.
    For each $k \in \N$, partition $B[0,1] = [-1,1]^{2}$ into rectangular regions of width $2U^{-k}$ and height $2V^{-k}$.
    From each rectangular region, remove $m$ rectangles of width $2U^{-(k+t)}$ and height $2V^{-(k+t)}$.
    Let $E_k$ be the collection of interiors of these rectangles, and set $E = \cup_{k \in \N} E_k$.
    The collection $\mathcal{RCO}(U,V,m,t)$ consists of all compact sets $F$ such that $F = B[0,1] \setminus E$, for a set $E$ constructed in this way.
    
    \begin{figure}[ht]
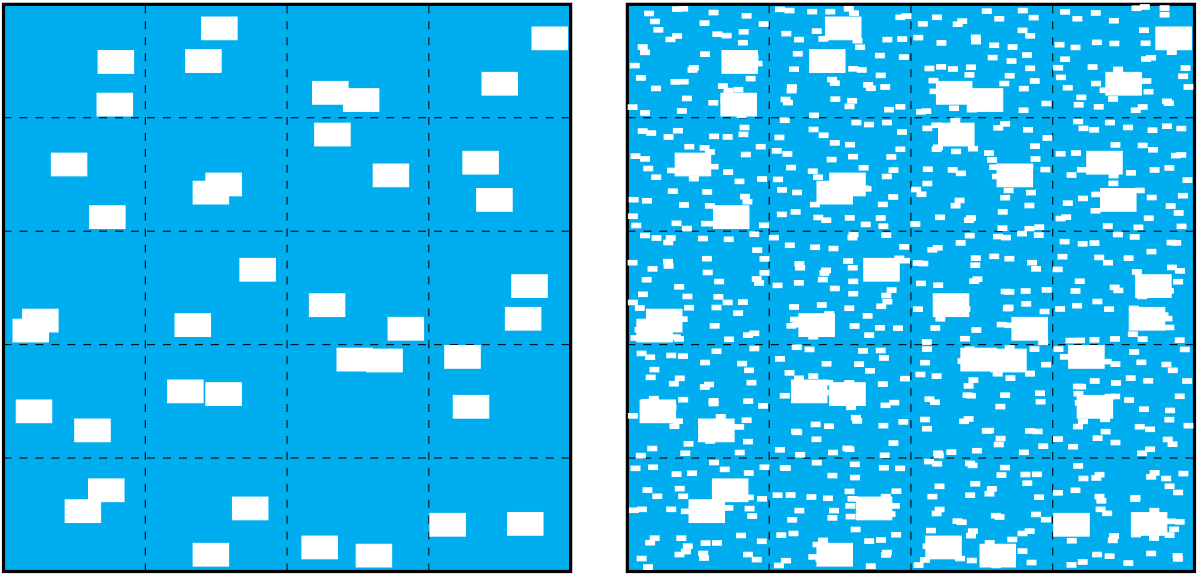

    \centering    \include{Pictures/RCO_4_5}
    \vspace{-1cm}
    \caption{The first two stages of construction ($F_{1}$ and $F_2$) of a set in $\mathcal{RCO}(4,5,2,1)$. The white portions are the rectangles in $E_1$ and $E_2$, respectively, that have been removed and
    the dashed lines represent the partition of $B[0,1] = [-1,1]^{2}$ used in the construction.}
    \label{Fig:Figure_Cut_Out_1}
    \end{figure}
    We highlight that overlapping cut-outs are permitted in our construction. For instance, this occurs in the example displayed in \Cref{Fig:Figure_Cut_Out_1}.
    \end{exmp}

\begin{prop}\label{prop:RCO-winning-cond}
    Let $U,V \in \N \setminus \{1\}$, $c \in (0,1)$, $m \in \N$ and $t>0$, and 
    \begin{align}\label{eq:RCO-alpha-value}
        \alpha (c) = (9m)^{1/c} (UV)^{-t} > 0.
    \end{align}
    For every $F \in \mathcal{RCO} (U,V,m,t)$, the set $F \cup (\R^2 \setminus B[0,1])$ is $(\alpha (c),A_{U,V},c,1,1)$-winning.
\end{prop}
\begin{proof}
    Let $F\in\mathcal{RCO}(U,V,m,t)$ and, for each $k \in \N$, let $E_k$ denote the cut-out sets in the construction of $F$, described in
    \Cref{exa:RCO}.
    For each $k \in \N$, let $\mathcal{Q}_k$
    denote the pairs of centre points and matrix powers defining the rectangles in the construction of $E_k$. 
    By construction, the collection $\mathcal{F}=\{\mathcal{Q}_1,\mathcal{Q}_2,\dots\}$ satisfies (SDIC~1) in the definition of a strategy defining iterative covering of $B[0,1]\setminus F$, with $A = A_{U,V}$.
    Hence, by \Cref{cutoutwin}, to show that $F \cup \R^{2}\setminus B[0,1]$ is winning, it remains to show (SDIC~2) holds.
    To this end, let $k \in \N$ be fixed, let $\rho_1=1$ and $c\in(0,1)$, and observe that, for $z\in\R^2$,
        \begin{align*}
        \#\{(q,y)\in\mathcal{Q}_k:(A_{U,V}^{k}(B[0,1])+z)\cap (A_{U,V}^{q}(B[0,1])+y)\neq\emptyset\}\leq 9m.
        \end{align*}
    Combining the above inequality with the fact that, if $(q,y)\in \mathcal{Q}_k$, then $q=k+t$, we have
    \begin{align*}
        \sum_{\substack{(q,y)\in \mathcal{Q}_{k}:\\(A^q(B[0,r])+y)\cap (A^k(B[0,\rho_1])+z)\neq \emptyset}} \left(\frac{1}{(UV)^q}\right)^c\leq 9m\left(\frac{1}{(UV)^k}\frac{1}{(UV)^t}\right)^c= \left(\frac{(9m)^{1/c}}{(UV)^t} \frac{1}{(UV)^k}\right)^c.
    \end{align*}
    Therefore, (SDIC~2) is satisfied with $a_k=(9m)^{1/c}/(UV)^t$ for all $k \in \N$, and so by \Cref{cutoutwin}, we conclude that the set $F\cup (\R^2 \setminus B[0,1])$ is $(\alpha(c),A_{U,V},c,1,1)$-winning.
\end{proof}

The second family of examples we consider is a class of Moran sets defined via affine transformations which, to the best of the authors' knowledge, do not fit into the current frameworks for topological and potential games.
This family includes examples with path connected complement and empty interior, so their thickness is zero.

\begin{exmp}\label{exa:RCD}
    Let $U,V\in\N\setminus\{1\}$ be given. 
    We define $\mathcal{RCD}(U,V)$ to be the collection of Moran sets $F = \cap_{k \in \N_0} F_k$, where $F_{0} =B[0,1]=[-1,1]^{2}$ and, for each $k \in \N$, $F_k$ consists of $(U-1)^k (V-1)^k$ rectangles, parallel to the (standard Euclidean) co-ordinate axes and with width $2U^{-k}$ and height $2V^{-k}$, defined inductively as follows.
    At the $k$-th stage of construction, each component rectangle of $F_k$ is split into $U-1$ equal sections along the $x$-axis and $V-1$ equal sections along the $y$-axis to create $(U-1)(V-1)$ smaller rectangular regions. 
    Each of these regions contains a component rectangle of $F_{k+1}$, positioned at one of the corners.
    The set $F_{k+1}$ is then defined to be the union of all such component rectangles.
    An illustration of the first two stages of construction ($F_{1}$ and $F_{2}$) of a set in $\mathcal{RCD} (7,4)$ is shown in \Cref{RCD}.

    \begin{figure}[ht]
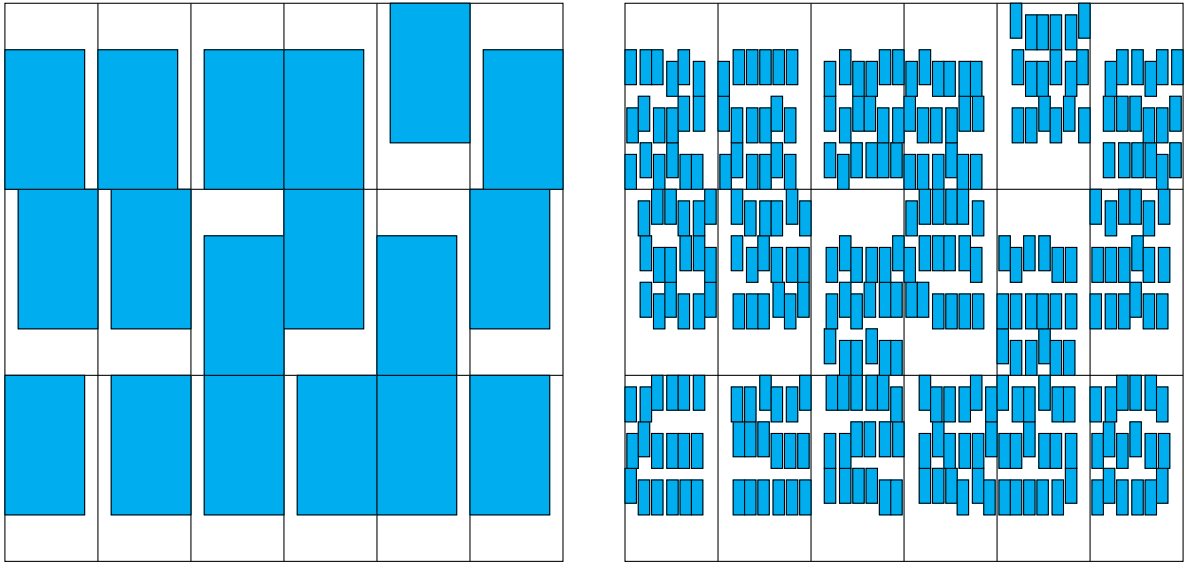

    \centering
    \include{Pictures/RCD_7_4}
    \vspace{-1cm}
    \caption{The first two stages of construction ($F_{1}$ and $F_{2}$) of a set in $\mathcal{RCD} (7,4)$.}
    \label{RCD}
    \end{figure}
    An alternative description for the sets in $\mathcal{RCD} (U,V)$ is as follows.  Define $ T_{\epsilon} = F_0 = B[0,1]$,
    for $k \in \N$, let $\Lambda_k$ denote the set $\{1,...,(U-1)(V-1)\}^k$, and let 
    $F_k = \cup_{\omega \in \Lambda_k} T_{\omega}$, where
    $T_{\omega} = A_{U,V}^{k}(B[0,1]
    ) + b_{\omega}$ for each $\omega \in \Lambda_k$, and where $b_{\omega}$ is defined inductively as follows.
    Set $b_{\epsilon} = (0,0)$ and
    suppose, for some $k\in\N_0$, that we have constructed $F_k = \cup_{\omega \in \Lambda_k} T_{\omega}$.
    For each $i\in\{1,...,(U-1)(V-1)\}$, there exists a unique $s \in \{1,\ldots,U-1\}$ and a unique $t \in \{1,\ldots,V-1\}$ such that $i=s+(t-1)(U-1)$.
    The vector $b_{\omega i}$ can be chosen to be any of the four vectors in the set
    \begin{align*}
        \bigg\{ b_{\omega}+ \bigg(\underbrace{\frac{2s-U}{U^{k}(U-1)},\frac{2t-V}{V^{k}(V-1)}}_{\mathfrak{u}_{s,k}}\bigg)+\bigg(\underbrace{\frac{(-1)^{q}}{U^{k+1}(U-1)},\frac{(-1)^{r}}{V^{k+1}(V-1)}}_{\mathfrak{v}_{q,r,k}}\bigg) \, : \, q, r \in \{0,1\}\bigg\}.
    \end{align*}
    Here, $b_{\omega}$ takes $A_{U,V}^{k+1}(B[0,1])$
    to the centre of $T_{\omega}$, the second vector $\mathfrak{u}_{s,k}$ takes $A_{U,V}^{k+1}(B[0,1]) + b_{\omega}$
    to the chosen component rectangle, and the vector $\mathfrak{v}_{q,r,k}$ takes $A_{U,V}^{k+1}(B[0,1]) + b_{\omega} + \mathfrak{u}_{s,k}$
    to the specific corner of the chosen component rectangle.
    \end{exmp}

    \begin{prop}\label{prop:RCD-winning-conditions}
        Let $U,V \in \N \setminus \{1\}$, $c \in (0,1)$ and $t>0$.
       If $F \in \mathcal{RCD} (U,V)$, then $F \cup (\R^2 \setminus B [0,1])$ is $(\alpha (c,t),A_{U,V},c,1,1)$-winning, where
        \begin{align}\label{eq:RCD-alpha-value}
            \alpha (c,t) = (9(U-1)(V-1)N_t)^{1/c} (UV)^{-(1+t)},
        \end{align}
        and
        \begin{align}\label{eq:Nt-value}
            N_{t}=\min\left\{\left\lceil \frac{U^{t}}{U-1} \right\rceil\left\lceil \frac{V^{t+1}}{V-1} \right\rceil+\left\lceil \frac{V^{t}}{V-1} \right\rceil\left\lceil U^{t} \right\rceil,\left\lceil \frac{U^{t+1}}{U-1} \right\rceil\left\lceil \frac{V^{t}}{V-1} \right\rceil+\left\lceil \frac{U^{t}}{U-1} \right\rceil\left\lceil V^{t} \right\rceil \right\}.
        \end{align}
    \end{prop}
    
    \begin{proof}
    Let $F\in\mathcal{RCD}(U,V)$ and let $t>0$ be fixed.
    By definition, $F=\cap_{k\in\N_0}\cup_{\omega\in \Lambda_k} T_{\omega}$, where $T_{\omega}$ is as defined above.   
    We construct a SDIC $\mathcal{F}$ of $B[0,r]\setminus F$ with respect to some tuple $(A_{U,V},(a_k)_{k\in\N},\rho_1,c)$.
    To this end, let $k\in\N_0$ and $\omega\in\Lambda_k$.
    For each $i\in\{1,\dots,(U-1)(V-1)\}$, we have that $T_{\omega i}$ is contained in a region $L_{\omega i}$ of width $2 U^{-k} (U-1)^{-1}$ and height $2 V^{-k} (V-1)^{-1}$. 
    Setting $q_{\omega i}=k+1+t$, we claim that one can cover $L_{\omega i}\setminus T_{\omega i}$, regardless of the four possible positions of $T_{\omega i}$ in $L_{\omega i}$, with $N_t$ sets of the form $A_{U,V}^{q_{\omega i}}(B[0,1])+y$.
    To verify this, observe that $L_{\omega i}\setminus T_{\omega i}$ can be partitioned into two regions $M_1$ and $M_2$, with $M_1$ having width $2(U^{k+1}(U-1))^{-1}$ and height either $2V^{-k+1}$ or $2(V^k(V-1))^{-1}$ and $M_2$ having height $2(V^{k+1}(V-1))^{-1}$ and width either $2U^{-k+1}$ or $2(U^k(U-1))^{-1}$, see \Cref{RCDreg}.

    \begin{figure}[ht]\label{PIC:Config}
    \centering
    \begin{tikzpicture}[x=0.45cm,y=0.4cm]
%%%% RECTANGLES %%%%
\draw[thick] (0,0) rectangle (10,10.5);
\draw[thick] (14,0) rectangle (24,10.5);
\draw[thick] (0,13) rectangle (10,23.5);
\draw[thick] (14,13) rectangle (24,23.5);

\draw[fill,cyan] (0,0) rectangle (7.5,8);
\draw[fill,cyan] (16.5,0) rectangle (24,8);
\draw[fill,cyan] (0,15.5) rectangle (7.5,23.5);
\draw[fill,cyan] (16.5,15.5) rectangle (24,23.5);
\draw[very thick] (0,0) rectangle (7.5,8);
\draw[very thick] (16.5,0) rectangle (24,8);
\draw[very thick] (0,15.5) rectangle (7.5,23.5);
\draw[very thick] (16.5,15.5) rectangle (24,23.5);

\draw[dashed] (7.5,8) -- (10,8);
\draw[dashed] (16.5,8) -- (16.5,10.5);
\draw[dashed] (14,15.5) -- (16.5,15.5);
\draw[dashed] (7.5,13) -- (7.5,15.5);
\draw (8.75,4.75) node {$M_1$};
\draw (8.75,18) node {$M_1$};
\draw (15.25,4.75) node {$M_1$};
\draw (15.25,18) node {$M_1$};
\draw (5,8.75) node {$M_2$};
\draw (5,14) node {$M_2$};
\draw (19,8.75) node {$M_2$};
\draw (19,14) node {$M_2$};

%%%% LABELS %%%%
\draw [<->] (0,-0.75) -- (7.5,-0.75);
\draw (3.75,-1.75) node {$\frac{2}{U^{k+1}}$};
\draw[<->] (7.5,-0.75) -- (10,-0.75);
\draw (9,-1.75) node {$\frac{2}{U^{k+1}(U-1)}$};
\draw[<->] (0,10.75) -- (10,10.75);
\draw (5,11.75) node {$\frac{2}{U^{k}(U-1)}$};
\draw [<->] (-0.5,0) -- (-0.5,8);
\draw (-2.75,4) node {$\frac{2}{V^{k+1}}$};
\draw [<->] (-0.5,8) -- (-0.5,10.5);
\draw (-3.5,9) node {$\frac{2}{V^{k+1}(V-1)}$};
\draw[<->] (10.5,0) -- (10.5,10.5);
\draw (12,5) node {$\frac{2}{V^{k}(V-1)}$};

\draw [<->] (16.5,-0.75) -- (24,-0.75);
\draw (21.25,-1.75) node {$\frac{2}{U^{k+1}}$};
\draw[<->] (14,-0.75) -- (16.5,-0.75);
\draw (15.5,-1.75) node {$\frac{2}{U^{k+1}(U-1)}$};
\draw[<->] (14,10.75) -- (24,10.75);
\draw (19,11.75) node {$\frac{2}{U^{k}(U-1)}$};
\draw [<->] (24.5,0) -- (24.5,8);
\draw (26.75,4) node {$\frac{2}{V^{k+1}}$};
\draw [<->] (24.5,8) -- (24.5,10.5);
\draw (27.5,9) node {$\frac{2}{V^{k+1}(V-1)}$};
\draw[<->] (13.5,0) -- (13.5,10.5);

\draw [<->] (0,24.25) -- (7.5,24.25);
\draw (3.75,25.25) node {$\frac{2}{U^{k+1}}$};
\draw[<->] (7.5,24.25) -- (10,24.25);
\draw (9,25.25) node {$\frac{2}{U^{k+1}(U-1)}$};
\draw[<->] (0,12.75) -- (10,12.75);
\draw [<->] (-0.5,15.5) -- (-0.5,23.5);
\draw (-2.75,19.5) node {$\frac{2}{V^{k+1}}$};
\draw [<->] (-0.5,13) -- (-0.5,15.5);
\draw (-3.5,14.5) node {$\frac{2}{V^{k+1}(V-1)}$};
\draw[<->] (10.5,13) -- (10.5,23.5);
\draw (12,18.5) node {$\frac{2}{V^{k}(V-1)}$};

\draw [<->] (16.5,24.25) -- (24,24.25);
\draw (21.25,25.25) node {$\frac{2}{U^{k+1}}$};
\draw[<->] (14,24.25) -- (16.5,24.25);
\draw (15.5,25.25) node {$\frac{2}{U^{k+1}(U-1)}$};
\draw[<->] (14,12.75) -- (24,12.75);
\draw [<->] (24.5,15.5) -- (24.5,23.5);
\draw (26.75,19.5) node {$\frac{2}{V^{k+1}}$};
\draw [<->] (24.5,13) -- (24.5,15.5);
\draw (27.5,14.5) node {$\frac{2}{V^{k+1}(V-1)}$};
\draw[<->] (13.5,13) -- (13.5,23.5);

% \filldraw (8.75,5.25) circle (1.75pt);
% \draw (8.75,4.75) node {\footnotesize{$y_{\omega|_k,i}$}};
% \filldraw (8.75,18.25) circle (1.75pt);
% \draw (8.75,17.75) node {\footnotesize{$y_{\omega|_k,i}$}};
% \filldraw (15.25,5.25) circle (1.75pt);
% \draw (15.25,4.75) node {\footnotesize{$y_{\omega|_k,i}$}};
% \filldraw (15.25,18.25) circle (1.75pt);
% \draw (15.25,17.75) node {\footnotesize{$y_{\omega|_k,i}$}};

% \filldraw (5,9.25) circle (1.75pt);
% \draw (5,8.75) node {\footnotesize{$z_{\omega|_k,i}$}};
% \filldraw (19,9.25) circle (1.75pt);
% \draw (19,8.75) node {\footnotesize{$z_{\omega|_k,i}$}};
% \filldraw (5,14.25) circle (1.75pt);
% \draw (5,13.75) node {\footnotesize{$z_{\omega|_k,i}$}};
% \filldraw (19,14.25) circle (1.75pt);
% \draw (19,13.75) node {\footnotesize{$z_{\omega|_k,i}$}};
\end{tikzpicture}
    \vspace{-1cm}
    \caption{The four possible configurations in a region of $\mathcal{RCD} (U,V)$.}
    \label{RCDreg}
    \end{figure}
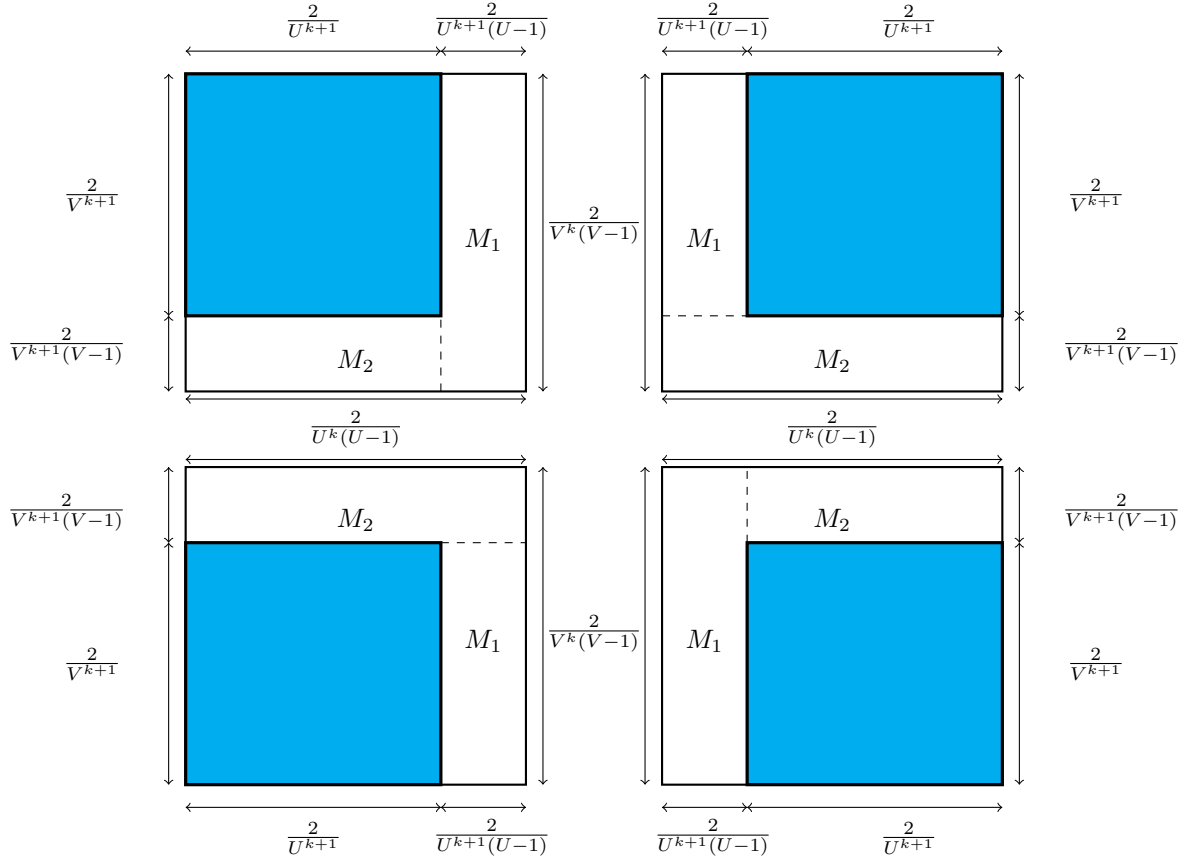

    To calculate how many such sets are needed to cover the regions $M_1$ and $M_2$ in total, we divide the respective side lengths of $M_{1}$ and $M_{2}$ by those of the rectangle $A_{U,V}^{q_{\omega i}}(B[0,1])=A_{U,V}^{k+1+t}(B[0,1])$ and sum them.
    Taking the minimum of the values for the two possible choices of $M_1$ and $M_2$ gives the number $N_t$ displayed in \eqref{eq:Nt-value}.
    Since $\omega\in\Lambda_k$ and $i\in \Lambda_{1}$ were chosen arbitrarily, we can repeat this process for each $\omega \in \Lambda_{k}$ and $i \in \Lambda_{1}$. 

    Setting $\mathcal{Q}_{\omega i}$ to be the set of tuples whose first co-ordinate is $q_{\omega i}$ and whose second co-ordinate is the centre of one of our covering sets for $L_{\omega i}\setminus T_{\omega i}$, we claim that $\mathcal{F}=\{\mathcal{Q}_1,\mathcal{Q}_2,\dots\}$, where $\mathcal{Q}_k=\cup_{\omega\in \Lambda_{k}}\cup_{i \in \Lambda_{1}} Q_{\omega i}$, is an SDIC of $B[0,1] \setminus F$.
    The condition (SDIC~1) is immediate from the construction, so it suffices to verify (SDIC~2).
    Let $k\in\N$ and $z\in\R^2$ be given. 
    Observe that $A_{U,V}^k(B[0,1])+z$ can intersect at most nine sets $T_{\omega}$ and for each $i\in \{1,\dots, (U-1)(V-1)\}$, that $L_{\omega i}\setminus T_{\omega i}$ can be covered by $N_{t}$ sets of the form $A_{U,V}^{q_{\omega i}}(B[0,1])$; hence,
        \begin{align*}
        \#\{ (q,y)\in \mathcal{Q}_{k} :(A_{U,V}^q(B[0,r])+y)\cap (A_{U,V}^k(B[0,1])+z)\neq \emptyset\}\leq 9(U-1)(V-1)N_{t}.
        \end{align*}
    Therefore, given $c\in (0,1)$ and setting $\rho_1=1$, we have
    \begin{align*}
        \sum_{\substack{(q,y)\in \mathcal{Q}_{k} :\\(A_{U,V}^q(B[0,r])+y)\cap (A_{U,V}^k(B[0,1])+z)\neq \emptyset}} \left(\prod_{j=1}^2\beta_{jj}^q\right)^c
        \leq \left(\left(9(U-1)(V-1)N_{t}\right)^{1/c}(UV)^{-(1+t)}\right)^c \left((UV)^{-k}\right)^c.
    \end{align*}    
    Hence (SDIC~2) is satisfied with $a_k=(9(U-1)(V-1)N_{t})^{1/c}(UV)^{-(1+t)}$, for all $k \in \N$, and $\rho_1=1$.
    Therefore, it follows by \Cref{cutoutwin} that
    the set $F\cup(\R^2\setminus B[0,1])$ is $(\alpha(c,t),A_{U,V},c,1,1)$-winning.
    \end{proof}

\section{Hausdorff dimension of winning sets}\label{sec:winng-sets-HD}

Here we present our main result, \Cref{thm:winning-set-dimension}, 
in which we provide lower bounds on the Hausdorff dimension of winning sets. We apply these bounds, in \Cref{sec:applications}, to prove that the intersection of certain classes of self-affine sets have strictly positive Hausdorff dimension. 

\begin{defin}
    Let $A$ be a diagonal ($n\!\times\!n$)-matrix with entries $\beta_{11},\dots, \beta_{nn} \in (0,1)$,
    $\rho_2>0$, $\alpha>0$ and $c\in[ 0,1)$. We define $\mathcal{S}(\alpha,A,c,\rho_2)$ to be the family of sets that are
    $(\alpha,A,c,\rho_2,\rho_1)$-winning
    for some $\rho_1\geq \rho_2$.
\end{defin}

\begin{thrm}\label{thm:winning-set-dimension}
    Given $n\in \N$, let $A$ be a diagonal $(n\!\times\!n)$-matrix with diagonal entries $\beta_{11},\dots,\beta_{nn} \in (0, 1/5)$, and let $\alpha, c\in (0,1)$ and $\rho_2>0$.
    If there exists a $\delta\in(0,1)$ with
    \begin{align}\label{eq:Thm-lower-bound-ass-1}
            \alpha^c\leq \delta^2\left(1-\left(\prod_{j=1}^n\beta_{jj}\right)^{1-c}\right) \quad \text{and} \quad 
            3^{-n}\prod\limits_{j=1}^n\left( 1-5\beta_{jj}^{\lfloor\delta\alpha^{-1}\rfloor}\right)>8^n (1+2^{2n+1})\delta,
    \end{align}
    then, for all $S\in \mathcal{S}(\alpha,A,c,\rho_2)$ and $y\in\R^n$, the intersection $S\cap \big(A(B[0,\rho_2])+y\big)$ is non-empty, and
    \begin{align*}
        \dim_{H}\left(S\cap \left(A(B[0,\rho_2])+y\right)\right)\geq\max\left\{ n-K_1\frac{\alpha}{|\log (\beta_{\max})|},0\right\},
    \end{align*}
    where $\beta_{\max} = \max_{j \in \{1,\ldots,n\}} \beta_{jj}$ and 
    \begin{align}\label{eq:main_result_constant_K}
        K_1=2\delta^{-1}\left\lvert\log\left( 3^{-n}\prod\limits_{j=1}^n\left( 1-5\beta_{jj}^{\lfloor\delta\alpha^{-1}\rfloor}\right)-8^n (1+2^{2n+1})\delta\right)\right\rvert.
    \end{align}
\end{thrm}

While the conditions of \Cref{thm:winning-set-dimension} may appear cumbersome to verify, we highlight that there exists a $\delta$ satisfying \eqref{eq:Thm-lower-bound-ass-1} if
   \begin{align*}
       \alpha^c\leq \left(\left(2 \cdot 8^{n} (1+2^{2n+1})\right)^{-1}3^{-n}\prod\limits_{j=1}^n\left( 1-5\beta_{jj}\right)\right)^2\left(1-\left(\prod_{j=1}^n\beta_{jj}\right)^{1-c}\right).
   \end{align*}
Specifically $\delta = (2 \cdot 8^{n} (1+2^{2n+1}))^{-1}3^{-n}\prod_{j=1}^n( 1-5\beta_{jj})$. However, as this $\delta$ may not be optimal, this may yield a weaker bound for the Hausdorff dimension.

\begin{cor}\label{cor:positive-HD}
    Assume the setting of \Cref{thm:winning-set-dimension} and that $\alpha < K_1^{-1} n \log \beta_{\max}^{-1}$.
    For all $y\in\R^n$, 
    \begin{align*}
    \dim_{H} \left(S\cap \left(A(B[0,\rho_2])+y\right) \right)>0.
    \end{align*}
\end{cor}

\begin{proof}[Proof of \Cref{thm:winning-set-dimension}]
    For ease of notation, we write $\rho$ in place of $\rho_2$. By our hypothesis, for some $\rho_1\geq \rho$, we have $S$ is $(\alpha,A,c,\rho,\rho_1)$-winning.
    By the rules of the game, Player\;I's first play has the form $U_1=A(B[0,\rho])+y$, for some $y\in\R^n$.
    For $k\in\N$, we define 
        \begin{alignat*}{3}
        &E_k = \frac{\rho}{2} A^{k} \left( \Z^n \right)+y, \quad &&  \quad &&\mathcal{E}_k= \{ A^{k}(B[0,\rho])+z:z\in E_k \},\\
        &D_k =  3\rho A^{k} \left(\Z^n\right)+y, \quad &&  \quad &&\mathcal{D}_k= \{ A^{k}(B[0,\rho])+z':z'\in D_k \}\subseteq \mathcal{E}_k,
        \end{alignat*}
    and set $N= \left\lfloor \delta\alpha^{-1} \right\rfloor$;
    note, $N \geq 1$ 
    by \eqref{eq:Thm-lower-bound-ass-1}.
    For $k \in \N$, we define $\pi_k:\mathcal{E}_{k+1}\rightarrow \mathcal{E}_k$ as follows:
    \begin{itemize}
        \item If $k\not\equiv 1 \mod N$, we set $\pi_k(A^{k+1}\left(B[0,\rho]\right)+z)$ to be the element of $\mathcal{E}_k$ containing $A^{k+1}\left(B[0,\rho]\right)+z$ such that the distance between $z$ and the centre point of $\pi_k(A^{k+1}\left(B[0,\rho]\right)+z)$ is minimised;
        if there exists more than one possible set, we make an arbitrary choice. Note that such an element always exists, since $\mathcal{E}_k$ is a covering of $\R^{n}$.
        \item If $k \equiv 1 \mod N$, we define $\pi_k(A^{k+1}\left(B[0,\rho]\right)+z)$ as the element of $\mathcal{D}_k$ containing $A^{k+1}\left(B[0,\rho]\right)+z$ if such an element exists;
        otherwise it is defined as above.
    \end{itemize}
    Given $k,m \in \N$ and $T \in \mathcal{E}_{k+1}$, we write $\pi_{k+1,m} (T) = \pi_m\circ\cdots\circ\pi_{k-1}\circ\pi_k(T)\in\mathcal{E}_m$ if $m\leq k$, and $\pi_{k+1,m} (T) = T$ otherwise.   
  
    \begin{claimA}\label{claim:claim_A}
        If $T \in \mathcal{D}_{kN+1}$ and $T' \in \mathcal{D}_{(k+1)N+1}$, for some $k \in \N_0$, and $z = (z_{1}, \dots, z_{n}) \in D_{kN+1}$ is such that \mbox{$T=A^{kN+1} (B[0,\rho])+z$} and $T' \subseteq 2^{-1} A^{kN+1} (B[0,\rho]) + z$, then $\pi_{(k+1)N+1,kN+1} (T')=T$.
    \end{claimA}

    \begin{proof}[Proof of \Cref{claim:claim_A}]
        \textls[-5]{By definition, $T' \subseteq \pi_{(k+1)N+1,kN+2} (T')$ and there exists $\tilde{z} = (\tilde{z}_{1},\dots, \tilde{z}_{n}) \in E_{kN+2}$ with}
        \begin{align*}
        \pi_{(k+1)N+1,kN+2} (T') = A^{kN+2} (B [0,\rho]) + \tilde{z}.
        \end{align*}
        Hence, $(A^{kN+2} (B [0,\rho]) + \tilde{z}) \cap 2^{-1} A^{kN+1} (B[0,\rho] + z) \neq \emptyset$, and so $\lvert z_{j}-\tilde{z}_{j}\rvert\leq 2^{-1} \rho\beta_{jj}^{kN+1}+\rho\beta^{kN+2}_{jj}$ for all $j \in \{1,\ldots,n\}$.
        Therefore, for all $x\in \pi_{(k+1)N+1,kN+2}(T')$, 
        \begin{align*}
            \lvert x_{j}-z_{j} \rvert
            \leq \lvert x_{j}-\tilde{z}_{j} \rvert + \lvert z_{j}-\tilde{z}_{j} \rvert
            \leq \rho\beta_{jj}^{kN+2} + (2^{-1} \rho \beta_{jj}^{kN+1}+\rho\beta^{kN+2}_{jj})
            \leq \rho\beta_{jj}^{kN+1}.
        \end{align*}
        Thus, $x \in T$ and so we have that $\pi_{(k+1)N+1,kN+2}(T')\subseteq T$.
        From this, we deduce that
        \[
        \pi_{(k+1)N+1,kN+1}(T')=\pi_{kN+1}(\pi_{(k+1)N+1,kN+2}(T'))= T.\qedhere
        \]
    \end{proof}
  
    Let $m\in\N$ and $T = A^{m}(B[0,\rho]) +z \in \mathcal{E}_m$ be given. Consider $T$ as Player\;I's $m$-th play, that is, let $U_{m} = T$, and consider the following finite sequence as Player\;I's previous plays, 
        \begin{align*}
        T \subseteq \pi_{m,m-1}(T)\subseteq\cdots \subseteq \pi_{m,1}(T).
        \end{align*}
    In other words, suppose that $\pi_{m,i}(T) = U_{i}$, for $i\in \{1, 2, \dots, m-1\}$. 
    Since Player\;II has a winning strategy, on their $m$-th turn, they select
        \begin{align*}
        \mathcal{A}(U_m,\pi_{m,m-1} (U_m),\ldots,\pi_{m,1}(U_m))=\left\{ (q_{i,m},y_{i,m}) : q_{i,m}
        \geq1 \; \text{and} \; y_{i,m}
        \in\R^n\right\}_{i\in I_m}.
        \end{align*}
    For $k<m$, we define
        \begin{align*}
        \mathcal{A}^*_k(U_m)=\left\{
        (q_{i,k},y_{i,k})
        \in\mathcal{A}(\pi_{m,k}(U_m),\dots,\pi_{m,1}(U_m)):(A^{q_{i,k}
        }(B[0,\rho])+ 
        y_{i,k}
        )\cap U_m\neq\emptyset \right\},
        \end{align*}
    and set
        \begin{align*}
        \phi_m(U_{m})= \begin{cases}
        0  & \text{if $m=1$,}\\
        \sum\limits_{t=1}^{m-1} \sum\limits_{(q_{i,m},y_{i,m})\in \mathcal{A}^*_t(U_{m})} \left(\prod\limits_{j=1}^n \beta_{jj}^{q_{i,m}}\right)^c & \text{otherwise.}  
        \end{cases}
        \end{align*}
    Further, for 
        $r, l \in \N$, with $r > m$, set
        \begin{align*}
        \mathcal{D}_{r}(T)= \left\{T'\in\mathcal{D}_{r} : T'\subseteq 2^{-1} A^m(B[0,\rho])+z\right\}
        \quad \text{and} \quad
        \mathcal{D}_{l}'= \left\{T\in\mathcal{D}_{l} :\phi_{l} (T)\leq \left(\delta\prod_{j=1}^{n} \beta_{jj}^{l} \right)^c \right\}.
        \end{align*} 
    Through the following claim, a technical counting argument which we prove in \Cref{Appendix}, for $k\in\N_0$, we obtain a useful lower bound for the cardinality of the intersection $\mathcal{D}_{(k+1)N+1}(T)\cap \mathcal{D}'_{(k+1)N+1}$.
    
    \begin{claimA}\label{claim:claim_B}
        For all $k\in\N_0$ and $T\in\mathcal{D}_{kN+1}'$, 
        \begin{align}\label{eq:ClaimB}
            \#\left( \mathcal{D}_{(k+1)N+1}(T)\cap \mathcal{D}_{(k+1)N+1}' \right)
            \geq \left(\prod_{j=1}^n\beta_{jj}^{-N}\right)\left( 3^{-n}\prod_{j=1}^{n} (1-5\beta_{jj}^N) -8^n (1+2^{2n+1})\delta\right) > 0.
        \end{align}
    \end{claimA}

    Given this bound, we construct the following set $F$, which we will show is a subset of our target set $S$.
    \begin{itemize}
        \item Let $\mathcal{B}_0 = \{ U_1 \}$, where $U_{1}$ is as defined at the start of our proof. 
        Observe that $U_1\in \mathcal{D}_1$ and $\phi_1(U_1)=0$.
        Therefore, $U_1\in \mathcal{D}_1'$, and hence $\mathcal{B}_0\subseteq \mathcal{D}_1'$.
        \item For $k\in\N_0$, given a collection $\mathcal{B}_{k}\subseteq \mathcal{D}_{kN+1}'$ we construct the next level $\mathcal{B}_{k+1}$ by replacing each element $T\in\mathcal{B}_k$ by         
            \begin{align}\label{eq:the_constant_K}
            K= \left\lceil\left(\prod_{j=1}^n\beta_{jj}^{-N}\right)\left( 3^{-n} \prod_{j=1}^{n} \left(1-5\beta_{jj}^N\right) -8^n (1+2^{2n+1})\delta\right)\right\rceil
            \end{align}
        elements in $\mathcal{D}_{(k+1)N+1}(T)\cap \mathcal{D}_{(k+1)N+1}'$.
        This is possible due to \Cref{claim:claim_B}.
    \end{itemize}
    We set
        \begin{align}\label{eq:sub_cantor_set}
            F = \bigcap_{k\in\N_0} \bigcup_{B \in \mathcal{B}_k} B.
        \end{align}
    By \eqref{eq:Thm-lower-bound-ass-1} and \Cref{claim:claim_B} we have that $\mathcal{B}_{k}$ is non-empty, by construction we have that $\mathcal{B}_k$ consist of closed sets, and by \Cref{claim:claim_A,claim:claim_B} and the definition of $\mathcal{B}_k$, to each $R \in \mathcal{B}_k$, there exists a $V \in \mathcal{B}_{k+1}$ with $V \subseteq R$. Thus, by Cantor's intersection theorem the set $F$ is non-empty and compact.    
    Moreover, $F$ falls into the class of sets considered in \cite[Section 2.4]{MR2581371} and \cite{zbMATH05003392}, thus we can obtain a lower bound on the Hausdorff dimension of $F$: see \Cref{Hausdorff} in the Appendix for more details.
    In particular, we obtain that 
    \begin{align*}
        \text{dim}_H(F)&\geq
        \max \left\{ n-K_1\frac{\alpha}{\lvert\log \beta_{\max}\rvert}, 0 \right\}
        \,\; \text{where} \,\;
        K_1=2\delta^{-1}\left\lvert\log\left( \prod\limits_{j=1}^n\left( \frac{1}{3}(1-5\beta_{jj}^N)\right)-8^n (1+2^{2n+1})\delta\right)\right\rvert.
        \end{align*}
    Thus, it suffices to show $F \subseteq S \cap U_{1}$. To this end, for a given $x\in F$, we build a sequence of plays $(U_k)_{k\in\N}$ for Player\;I with $\cap_{k \in \N} U_k=\{x\}$ and so that the winning strategy ensures $x\in S$.
    
    For each $x \in F$, by the definition of $F$ and \Cref{claim:claim_A}, there exists a unique nested sequence of sets $(T_{kN+1})_{k\in\N_0}$ such that $x \in T_{kN+1}\in\mathcal{B}_{k} \subseteq \mathcal{D}_{kN+1}(T_{(k-1)N+1})\cap \mathcal{D}_{kN+1}'$ with $\pi_{kN+1,(k-1)N+1}(T_{kN+1})=T_{(k-1)N+1}$, for all $k \in \N$.
    For $l \in \N$ with $l \equiv 1 \operatorname{mod} N$, we set $U_{l}=T_{l}$ and for $l \in\N$ with $l \not\equiv 1 \operatorname{mod} N$, we let $\ell = \ell_{l}$ be the smallest natural number such that $l \leq \ell N+1$ and set $U_{l}= \pi_{\ell N+1,l}(T_{\ell N+1})$.
    By construction, for $l\in\N$, this sequence has the property that $U_{l}\in \mathcal{E}_{l}$, $U_{l}=\pi_{l}(U_{l+1})$ and if $l=kN+1$ for some $k \in \N_0$, then $U_{l} \in \mathcal{B}_{k}$.
    It can readily be verified that this sequence $(U_{l})_{l\in \N}$ constitutes a possible sequence of plays by Player\;I with outcome $x$.

    Observe that $F \subseteq U_{1}$ and fix $x \in F$.
    Assume for a contradiction that $x \not\in S$.
    Let $(U_{l})_{l\in\N}$ be the sequence of plays described above with $x$ the outcome of the game.
    Since Player\;II has a winning strategy and since we have assumed that $x\not\in S$, 
        \begin{align*}
        x \in \bigcup_{m\in\N} \bigcup_{\substack{(q_{i,m},y_{i,m})\\\in\mathcal{A}(U_m,\dots,U_1)}} \left(A^{q_{i,m}}(B[0,r])+y_{i,m}\right).
        \end{align*}
    Therefore, $x \in A^{q_{i,m}}(B[0,r])+y_{i,m}$ for at least one $m\in\N$ and $(q_{i,m},y_{i,m})\in\mathcal{A}(U_m,\dots,U_1)$.
    Moreover, since $x \in U_{l}$ for all $l\in\N$, we have $x\in(A^{q_{i,m}}(B[0,r])+y_{i,m})\cap U_{l}$ for all $l \geq m$ and therefore $(q_{i,m},y_{i,m})\in\mathcal{A}_m^*(U_{l})$.
    Hence, for every $k \in \N$ with $kN+1> m$,
        \begin{align*}
        \phi_{kN+1}(U_{kN+1})\geq \left(\prod\limits_{j=1}^n \beta_{jj}^{q_{i,m}} \right)^c.
        \end{align*}
    However, for $k \in \N$, we have $\phi_{kN+1}(U_{kN+1}) \leq (\delta\prod_{j=1}^n\beta_{jj}^{kN+1})^c$, since $U_{k N+1}\in \mathcal{B}_{k}\subseteq \mathcal{D}_{k N+1}'$, and so
        \begin{align*}
        \left(\prod\limits_{j=1}^n \beta_{jj}^{q_{i,m}} \right)^c\leq \left(\delta\prod\limits_{j=1}^n\beta_{jj}^{kN}\right)^c
        \end{align*}    
    for all $k \in \N$ with $kN+1\geq m$.
    Letting $k$ tend to infinity, we obtain that $(\prod_{j=1}^n \beta_{jj}^{q_{i,m}})^c=0$, yielding a contradiction to our hypothesis that $\beta_{jj} \in (0, 1/5)$ for $j \in \{ 1, \dots, n \}$. 
    Therefore, $F \subseteq S \cap U_1$,
    and so $\dim_H(S\cap U_1) \geq \dim_H(F) \geq n - K_{1}\alpha/\lvert\log(\beta_{\max})\rvert$.
\end{proof}

\section{Applications to patterns and intersections}\label{sec:applications}

We conclude by presenting several applications of \Cref{thm:winning-set-dimension}.
In \Cref{subsec:patterns}, we provide sufficient conditions for the existence of patterns in winning sets, and in \Cref{subsec:intersections}, we provide a lower bound for the Hausdorff dimension of the intersection of a collection of winning sets, in particular those sets constructed in \Cref{exa:RCO,exa:RCD}, which includes a large class of self-affine sets.

\subsection{Patterns in winning sets}\label{subsec:patterns}

We say that a set $S$ \textit{contains a homothetic copy of} $C=\{b_1,...,b_M\}$, with $M \in \N$, often called a \textit{pattern}, if there exists $\lambda >0$ such that $\cap_{i=1}^M (S-\lambda b_i)\neq\emptyset$.

In \cite{falconer2022intersections-ThicknessBasedGame,OriginalAlexiaPaper,yavicoli2022thickness-BoundaryBasedGame} the potential game was used to verify the existence of patterns within classes of compact sets with a given thickness.
In fact, these results are stronger than simply containing patterns, they verify that there exists a homothetic copy of every finite set of a certain cardinality.
We prove a similar result for winning sets of the matrix potential game.
This allows us to prove the existence of patterns in certain self-affine sets, to which the results in \cite{falconer2022intersections-ThicknessBasedGame,OriginalAlexiaPaper,yavicoli2022thickness-BoundaryBasedGame}, to authors' knowledge, cannot be applied.

\begin{thrm}\label{thm:patterns}
    Given $n \in \N$, let $A$ be a diagonal ($n\!\times\!n$)-matrix with diagonal entries $\beta_{11},\dots,\beta_{nn} \in (0, 1/5)$, $\beta_{\max} = \max_j \beta_{jj}$, $\alpha,c\in(0,1)$ and $\rho_2>0$.
    Further, assume there exist $\delta\in(0,1)$ and $M \in \N$ such that
        \begin{align}\label{pattern_inequal_1}
            M\alpha^c\leq \delta^{2}\left(1-\left(\prod_{j=1}^n\beta_{jj}\right)^{1-c}\right) \quad \text{and} \quad 
            3^{-n} \prod\limits_{j=1}^{n} (1-5\beta_{jj}^{\lfloor\delta(M^{1/c}\alpha)^{-1}\rfloor}) > 8^n (1+2^{2n+1})\delta.
        \end{align} 
    For all $S\in\mathcal{S}(\alpha,A,c,\rho_2)$, given a finite non-empty set $C$ with at most $M$ elements, $y\in\R^n$ and $\lambda\in (0,\rho_2(1-\beta_{\max})/\text{diam}(C))$, there exists a non-empty set $X\subseteq \R^n$ such that $\lambda C+x\subseteq S\cap B[y,\rho_2]$ for all $x\in X$.
    Moreover, if
    \begin{align*}
        M\alpha^c\leq \min\{\delta^2,K_{M}^{-1} n \log \beta_{\max}^{-1}\}\left(1-\left(\prod_{j=1}^n\beta_{jj}\right)^{1-c}\right),
    \end{align*}
    where 
        \begin{align*}
        K^{}_{M} =2\delta^{-1}\left\lvert \log\left( 3^{-n}\prod\limits_{j=1}^{n} (1-5\beta_{jj}^{\lfloor\delta(M^{\frac{1}{c}}\alpha)^{-1}\rfloor}) - 8^n (1+2^{2n+1})\delta\right)\right\rvert,
        \end{align*}
    then $\dim_{H}(X) \geq n-K^{}_{M} \alpha (\lvert\log (\beta_{\max}) \rvert)^{-1}>0$.
\end{thrm}

\begin{proof}
    Let $y\in\R^n$ be given and assume without loss of generality that $C$ is a non-empty set with exactly $M$ elements. 
    Let $\lambda\in(0,\rho_2(1-\beta_{\max})/\text{diam}(C))$ be given.
    Denote $\lambda C=\{a_1,\dots,a_M\}$ and assume, without loss of generality, since we may apply a translation by $y-a_{1}$ to $X$,
    that $a_1=y$, so 
    that $\lambda C\subseteq B[y,\rho_2(1-\beta_{\max})]$.
    Let $S\in\mathcal{S}(\alpha,A,c,\rho_2)$ and, for each $i\in \{1,...,M\}$, write $S_i= S\cap B[y,\rho_2]-a_i$.
    Observe that if the set
    \begin{align*}
    X=\left(A(B[0,\rho_2])+y\right)\cap\bigcap\limits_{i=1}^M S_i
    \end{align*}
    is non-empty, then for all $x \in X$ and $i \in \{1,\ldots,M\}$, we have $x + a_i \in S_i + a_i = S \cap B[y,\rho_2]$, and so it follows $x + \lambda C \subseteq S \cap B [y,\rho_2]$.
    Thus, it suffices to show the set $X$ is non-empty.
    To this end, observe that there exists some $\rho_1\geq \rho_2$ such that $S$ is $(\alpha,A,c,\rho_2,\rho_1)$-winning.
    Thus, for all $i \in \{1,\ldots,M\}$,
    the set $S_i\cup(\R^n\setminus B[y-a_i,\rho_2])$ is $(\alpha,A,c,\rho_2,\rho_1)$-winning, since it is a superset of $S - a_i$ and so one can play the strategy for $S$ but translated.
    By the countable intersection property (Lemma~\ref{lem:countableintersections}) 
    we deduce $\cap_{i} (S_i\cup (\R^n\setminus B[y-a_i,\rho_2]))$ is $(M^{1/c}\alpha,A,c,\rho_2,\rho_1)$-winning.
    Therefore, Theorem~\ref{thm:winning-set-dimension} yields 
    \begin{align*}
         \emptyset
         &\neq \bigg(A(B[0,\rho_2])+y\bigg)\cap \left( \bigcap\limits_{i=1}^M S_i\cup (\R^n\setminus B[y-a_i,\rho_2])\right) \\
         &\subseteq \bigg(A(B[0,\rho_2])+y\bigg)\cap \left( \bigcap\limits_{i=1}^M S_i \cup (\R^n\setminus \bigcap\limits_{i=1}^M B[y-a_i,\rho_2]) \right)\\
         &=X \cup \left(\bigg(A(B[0,\rho_2])+y\bigg)\cap  (\R^n\setminus \bigcap\limits_{i=1}^M B[y-a_i,\rho_2])\right).
    \end{align*}
    The second term in the union in the final line is empty since, as $\lambda < \rho_2 (1- \beta_{\max}) /\operatorname{diam} (C)$, we have
    \begin{align*}
        \bigcap\limits_{i=1}^M \bigg(B[y-a_i,\rho_2]\bigg)\supseteq B[y,\rho_2-\rho_2(1-\beta_{\max})]=B[y,\rho_2 \beta_{\max}] \supseteq A(B[0,\rho_2])+y.
        \end{align*}
    Thus, $X$ is non-empty and so $x + \lambda C \subseteq S \cap B[y,\rho_2]$ for all $x \in X$.
    The lower bound for $\dim_{H}(X)$ follows by application of \Cref{cor:positive-HD}.
\end{proof}

Given the setup of \Cref{thm:patterns}, we let $M (\alpha, A, c, \rho_2)$ denote the largest natural number such that there exists a $\delta$ satisfying  \eqref{pattern_inequal_1}.

As a consequence of \Cref{thm:patterns}, we obtain the following application to distance sets.
Namely, that winning sets with sufficiently strong winning conditions contain an interval of distances on each axis within every sufficiently large closed ball.
For given $\eta_1,\dots,\eta_n\in(0,1)$, the result follows immediately from \Cref{thm:patterns} by considering the set $C= \{(0,\dots,0),(\eta_1,\dots,\eta_n)\}$.

\begin{cor}\label{cor:cor_6.2}
    Given $n \in \N$, let $A$ be a diagonal ($n\!\times\!n$)-matrix with diagonal entries $\beta_{11},\dots,\beta_{nn} \in (0, 1/5)$, let $\alpha,c\in (0,1)$, and let $\rho>0$.
    If $M\left(\alpha,A,c,\rho_2\right)\geq 2$, then given $\eta_1,\dots, \eta_n\in [0,1]$ and $\lambda\in \left(0,(1-\beta_{\max})\rho_2\right)$, for all $y\in\R^n$ and $S\in\mathcal{S}(\alpha,A,c,\rho_2)$, we have $S\cap B[y,\rho_2]$ contains two points $x=(x_{1},\dots,x_{n}), y=(y_{1},\dots,y_{n})$ satisfying $|x_{j}-y_{j}|=\eta_j\lambda$ for all $j \in \{1,\dots,n\}$.
\end{cor}

We now demonstrate how \Cref{thm:patterns} can be used to show that for all $M \in \N$, there exist totally disconnected, self-affine sets containing a homothetic copy of every set with at most $M$ elements.
This extends an implication of \cite[Theorem 20]{yavicoli2022thickness-BoundaryBasedGame}, which establishes the existence of dust-like \emph{self-similar} sets with this property.
In fact, we can provide explicit examples of such sets in the family $\mathcal{RCD} (U,V)$ defined in Example~\ref{exa:RCD}, for sufficiently large $U,V \in \N$.
\begin{prop}\label{prop:self-affine-patterns}
    Let $M\in\N$ and $\ell\in\N_0$.
    There exists $U\in\N$ such that every $F\in \mathcal{RCD}(U,U+\ell)$ contains a homothetic copy of every set with at most $M$ elements.
\end{prop}

\begin{proof}
    Setting $c=9/10$, by Proposition~\ref{prop:RCD-winning-conditions}, for all $U\in\N\setminus \{1\}$ and $F\in \mathcal{RCD}(U,U+\ell)$ the set $F \cup (\R^2 \setminus B [0,1])$ is $(\alpha^{}_U (c,1),A_{U,U+\ell},c,1,1)$-winning, where
        \begin{align*}
            \alpha^{}_U (c,1) = (9(U-1)(U+\ell-1)N_1)^{1/c} (U(U+\ell))^{-2},
        \end{align*}
        with
        \begin{align*}
            N_{1}&=\min\left\{\left\lceil \frac{U}{U-1} \right\rceil\left\lceil \frac{(U+\ell)^{2}}{U+\ell-1} \right\rceil+\left\lceil \frac{(U+\ell)}{U+\ell-1} \right\rceil\left\lceil U \right\rceil,\left\lceil \frac{U^{2}}{U-1} \right\rceil\left\lceil \frac{(U+\ell)}{U+\ell-1} \right\rceil+\left\lceil \frac{U}{U-1} \right\rceil\left\lceil (U+\ell) \right\rceil \right\}\\
            &= 2(U+2)+2(U+\ell)=2(2U+\ell+2).
        \end{align*}
        Therefore, it follows that 
        \begin{align*}
            \alpha^{}_U \left(9/10,1\right) &= (18(U-1)(U+\ell-1)(2U+\ell+2))^{1/c} (U(U+\ell))^{-2}\\
            &\leq \left( 36+\frac{36}{U+\ell} \right)^{10/9}U^{-6/9}\left(1+\frac{\ell}{U}\right)^{2/9}.
        \end{align*}
        Observing that $ \alpha^{}_U \left(9/10,1\right) \rightarrow 0$ as $U \rightarrow \infty$,
        we have that, for a given $M\in\N$, the inequalities in \eqref{pattern_inequal_1} are satisfied for all sufficiently large $U$.
        Hence, applying Theorem~\ref{thm:patterns}, there exists sufficiently large $U\in\N$ such that for each set $F\in \mathcal{RCD}(U,U+\ell)$ the set $(F \cup (\R^2 \setminus B [0,1]))\cap B[0,1]=F$ contains a homothetic copy of every set with at most $M$ elements.
\end{proof}

Since, for all $U,V \in \N$, the family $\mathcal{RCD} (U,V)$ contains a totally disconnected attractor of an IFS, we obtain the following consequence of \Cref{prop:self-affine-patterns}.

\begin{cor}
    Given $M\in \N$, there exists a totally disconnected attractor of an iterated function system consisting of strictly affine maps that contains a homothetic copy of every set with at most $M$ elements.
\end{cor}

We now apply \Cref{thm:patterns} to demonstrate the existence of patterns for some explicit examples of sets in the families introduced in \Cref{sec:winning-sets}.
We highlight that the techniques developed in \cite{falconer2022intersections-ThicknessBasedGame,yavicoli2022thickness-BoundaryBasedGame} cannot be applied here as these sets have zero thickness.

\begin{exmp}
    Given $U,V \in \N \setminus \{1\}$, $\ell>0$ and $m\in\N$, let $\mathcal{RCO} (U,V,m,\ell)$ be the family of sets defined in \Cref{exa:RCO}.
    Using \Cref{thm:patterns}, we can show that certain patterns appear in every $E \in \mathcal{RCO} (U,V,m,\ell)$, depending on the values of $U,V,m$ and $\ell$.
    By \Cref{prop:RCO-winning-cond}, the set $E \cup (\R^2 \setminus B[0,1])$ is $(\alpha(c),A_{U,V},c,1,1)$-winning for all $c \in (0,1)$, where $\alpha(c)$ is as defined in \eqref{eq:RCO-alpha-value}, and $A_{U, V}$ is as in \eqref{eq:AUV}.
    Thus, to apply \Cref{thm:patterns} and to determine if $E \cup (\R^2 \setminus B[0,1])\cap B[0,1]=E$ contains patterns, it suffices to show that there exists a $\delta \in (0,1)$ such that the inequalities in \eqref{pattern_inequal_1} hold, with $\alpha = \alpha (c)$.
    Using numerical methods to find the optimum values of $c$ and $\delta$, we obtain the following.
    \begin{itemize}
    \itemsep0.5em
        \item For $U=12, V=15, m=1$ and $\ell=5$, every $F \in \mathcal{RCO} (U,V,m,\ell)$ contains a homothetic copy of every set with at most $4$ elements. 
        Moreover, for a set with at most $4$ elements, the associated set $X$ has $\dim_{H}(X) \geq 1.99996$.
        \item For $U=17, V=24, m=1$ and $\ell=5$, every $F \in \mathcal{RCO} (U,V,m,\ell)$ contains a homothetic copy of every set with at most $232$ elements. 
        Moreover, for a set with at most $232$ elements, the associated set $X$ has $\dim_{H}(X) \geq 1.99997$.
        \item For $U=271828$, $V=314159$, $m=2$ and $\ell=1$, every $F \in \mathcal{RCO} (U,V,m,\ell)$ contains a homothetic copy of every set with at most $3$ elements. 
        Moreover, for a set with at most $3$ elements, the associated set $X$ has
        $\dim_{H}(X) \geq 1.99997$.
    \end{itemize}
\end{exmp}

\begin{exmp}
    Given $U,V \in \N \setminus \{1\}$, let $\mathcal{RCD} (U,V)$ be the family of sets defined in \Cref{exa:RCD}.
    By \Cref{prop:RCD-winning-conditions}, for all $E \in \mathcal{RCD} (U,V)$, the set $E \cup (\R^2 \setminus B[0,1])$ is $(\alpha (c,t),A_{U,V},c,1,1)$-winning for all $c \in (0,1)$ and $t>0$, where $\alpha (c,t)$ is as defined in \eqref{eq:RCD-alpha-value}, and $A_{U,V}$ is as in \eqref{eq:AUV}.
    As with our previous example, to apply \Cref{thm:patterns} and to determine if $E \cup (\R^2 \setminus B[0,1])\cap B[0,1]=E$ contains patterns, it suffices to show that there exists a $\delta \in (0,1)$ such that the inequalities in \eqref{pattern_inequal_1} hold, with $\alpha = \alpha (c,t)$.
    Using numerical methods 
    we obtain the following.
    \begin{itemize}
    \itemsep0.5em
        \item For $U=2^{37}$ and $V=2^{38}$, every $F \in \mathcal{RCD} (U,V)$ contains a homothetic copy of every set with at most $4$ elements. 
        Moreover, for a set with at most $4$ elements, the associated set $X$ has $\dim_{H}(X) \geq 1.99999$.
        \item For $U=900019043105$ and $V=999921083009$, every $F \in \mathcal{RCD} (U,V)$ contains a homothetic copy of every set with at most $21$ elements. 
        Moreover, for a set with at most $21$ elements, the associated set $X$ has $\dim_{H}(X) \geq 1.99999$.
        \item Given an arbitrary set $F\in \mathcal{RCD}(900019043105,999921083009)$ and sets $\{E_k\}_{k=1}^5$ such that \mbox{$E_k \in\mathcal{RCO}(900019043105,999921083009,4,k)$}, we have that $E_1\cap E_2\cap E_3\cap E_4\cap E_5\cap F$ contains a homothetic copy of every set with at most $4$ elements.
        Moreover, for a set with at most $4$ elements, the associated set $X$ has $\dim_{H}(X) \geq 1.99999$.
    \end{itemize}
\end{exmp}
We highlight that despite the defining parameters in the previous examples being very large, this is still a strong result.
Indeed, previous techniques have not been able to study the
sets in $\mathcal{RCD}(U,V)$ when $U\neq V$ and even when $U=V$ the techniques only work for attractors with strong separation and exact values depend on constants that are known to exist but whose values are unknown.

\subsection{Dimensions of intersections}\label{subsec:intersections}

As a consequence of \Cref{thm:winning-set-dimension} and the countable intersection property (\Cref{lem:countableintersections}), we obtain a lower bound on the Hausdorff dimension of the intersection of a collection of winning sets.

\begin{cor}\label{cor:intersections}
    Given $n \in \N$, let $A$ be a diagonal ($n\!\times\!n$)-matrix with diagonal entries $\beta_{11},\dots,\beta_{nn} \in (0, 1/5)$, let $c\in (0,1)$, and let $\rho>0$.
    Moreover, let $I$ be an at most countable set and, for each $i \in I$, let $\alpha_i \in (0,1)$ and $S_i\in \mathcal{S}(\alpha_i,A,c,\rho_2)$.
    Set $S = \cap_{i \in I} S_i$.
    If there exists $\alpha \in (0,1)$ with $\alpha^c=\sum_{i \in I}\alpha_i^c$ and $\delta\in(0,1)$ such that
        \begin{align}\label{pattern_inequal_2}
            \alpha^c\leq \delta^2\left(1-\left(\prod_{j=1}^n\beta_{jj}\right)^{1-c}\right)<1 \quad \text{and} \quad 0<\prod\limits_{j=1}^n\left( \frac{1}{3}(1-5\beta_{jj}^{\lfloor\delta\alpha^{-1}\rfloor})\right)-8^n (1+2^{2n+1})\delta,
        \end{align}
        then, for every $y\in\R^n$, we have that $S\cap \big(A(B[0,\rho_2])+y\big)\neq \emptyset$.
        Moreover, letting $K_1$ be as defined in \eqref{eq:main_result_constant_K},
    \begin{align*}
        \text{dim}_H\left(S\cap \big(A(B[0,\rho_2])+y\big)\right)\geq\max\left\{ n-K_1\frac{\alpha}{\left|\log(\beta_{\max})\right|},0\right\}.
    \end{align*}
\end{cor}

This result has several consequences.
Indeed, when considering a specific family $\mathcal{S}(\alpha,A,c,\rho_2)$, we have, given any $M(\alpha,A,c,\rho_2)$ sets in  $\mathcal{S}(\alpha,A,c,\rho_2)$, that they have non-empty intersection.
This can be seen as there exists $\delta\in(0,1)$ with $\delta$ and $M(\alpha,A,c,\rho_2)$ satisfying the inequalities in \eqref{pattern_inequal_1}, and hence also \eqref{pattern_inequal_2}.

With the use of the above intersection results one can study intersections of winning sets with the dimensional estimates provided in \Cref{thm:winning-set-dimension}. Indeed, the intersection results provided in \cite{falconer2022intersections-ThicknessBasedGame} have been recently used in \cite{baker2025finitelybaseqexpansions} and \cite{yavicoli2025numbersomittingdigitscertain} to provide dimensional bounds on intersections of sets arising from certain number expansions.
We now provide some examples of our results applied to intersections of sets from \Cref{exa:RCO,exa:RCD}.

\begin{exmp}
    Using \Cref{cor:intersections}, we obtain that sets in $\mathcal{RCD}(U,V)$ and $\mathcal{RCO}(U,V,m,t)$ have large intersections, when the values of $U$ and $V$ are sufficiently large, and $m$ and $t$ are chosen appropriately.
    Using numerical methods, we obtain the following.
    \begin{itemize}
    \itemsep0.5em
        \item If $U=2^{37}$ and $V=2^{36}$, then for all $F_1,F_2\in \mathcal{RCD}(U,V)$, $E_1\in \mathcal{RCO}(U,V,1,2)$ and $E_2\in\mathcal{RCO} (U,V,1,6)$, we have that $\dim_H(F_1\cap F_2\cap E_1\cap E_2) \geq1.999993$.
        \item If $U=2^{36}$ and $V=2^{40}$, then for all $F\in \mathcal{RCD}(U,V)$ and $E\in\mathcal{RCO}(U,V,1,1)$, we have that
        $\dim_H(F\cap E)\geq 1.999997$.
        \item If $U=425$ and $V=365$, then for all $F\in \mathcal{RCO}(U,V,10,3)$ and $E\in\mathcal{RCO}(U,V,1,2)$, we have that
        $\dim_H(F\cap E)\geq 1.99998$.
    \end{itemize}
\end{exmp}

\section*{Acknowledgements}

The authors thank Ricky Hutchins for valuable discussions.
RAH thanks the Engineering and Physical Sciences Research Council (EPSRC) Doctoral Training Partnership and the University of Birmingham for financial support.
AM and TS were supported by EPSRC grant EP/Y023358/1, and TS was additioanlly supported by EPSRC grant UKRI266. AM also acknowledges support from an EPSRC Doctoral Prize Fellowship, funded by the Open University's Doctoral Training Partnership grant EP/T518165/1.

\bibliographystyle{abbrv}
\bibliography{Ref.bib}

@article {broderick2017quantitative-OriginalRestrictedPotential,
    AUTHOR = {Broderick, Ryan and Fishman, Lior and Simmons, David},
     TITLE = {Quantitative results using variants of {S}chmidt's game:
              dimension bounds, arithmetic progressions, and more},
   JOURNAL = {Acta Arith.},
  FJOURNAL = {Acta Arithmetica},
    VOLUME = {188},
      YEAR = {2019},
    NUMBER = {3},
     PAGES = {289--316},
}

@article{falconer2022intersections-ThicknessBasedGame,
      AUTHOR = {Falconer, Kenneth and Yavicoli, Alexia},
     TITLE = {Intersections of thick compact sets in {${\mathbb{R}}^d$}},
   JOURNAL = {Math. Z.},
  FJOURNAL = {Mathematische Zeitschrift},
    VOLUME = {301},
      YEAR = {2022},
    NUMBER = {3},
     PAGES = {2291--2315},
}

@article {yavicoli2022thickness-BoundaryBasedGame,
    AUTHOR = {Yavicoli, Alexia},
     TITLE = {Thickness and a gap lemma in {$\mathbb{R}^d$}},
   JOURNAL = {Int. Math. Res. Not. IMRN},
  FJOURNAL = {International Mathematics Research Notices. IMRN},
      YEAR = {2023},
    VOLUME = {19},
     PAGES = {16453--16477},
}

@book {TechniquesInFractalGeometry,
    AUTHOR = {Falconer, Kenneth},
     TITLE = {Techniques in fractal geometry},
 PUBLISHER = {John Wiley \& Sons, Ltd., Chichester},
      YEAR = {1997},
}

@book{FractalGeometry,
  author = {Falconer, Kenneth},
  publisher = {John Wiley \& Sons, Ltd., Chichester},
  title = {Fractal geometry - mathematical foundations and applications.},
  year = {1990}
}

@article {OriginalAlexiaPaper,
    AUTHOR = {Yavicoli, Alexia},
     TITLE = {Patterns in thick compact sets},
   JOURNAL = {Israel J. Math.},
  FJOURNAL = {Israel Journal of Mathematics},
    VOLUME = {244},
      YEAR = {2021},
    NUMBER = {1},
     PAGES = {95--126},
}

@article {MR2581371,
    AUTHOR = {Kleinbock, Dmitry and Weiss, Barak},
     TITLE = {Modified {S}chmidt games and {D}iophantine approximation with
              weights},
   JOURNAL = {Adv. Math.},
  FJOURNAL = {Advances in Mathematics},
    VOLUME = {223},
      YEAR = {2010},
    NUMBER = {4},
     PAGES = {1276--1298},
}

@article {MR3826896,
    AUTHOR = {Fishman, Lior and Simmons, David and Urba\'nski, Mariusz},
     TITLE = {Diophantine approximation and the geometry of limit sets in
              {G}romov hyperbolic metric spaces},
   JOURNAL = {Mem. Amer. Math. Soc.},
    VOLUME = {254},
      YEAR = {2018},
}

@article {Arithmetic_progressions_in_sets_of_fractional_dimension,
    AUTHOR = {{\L}aba, Izabella and Pramanik, Malabika},
     TITLE = {Arithmetic progressions in sets of fractional dimension},
   JOURNAL = {Geom. Funct. Anal.},
  FJOURNAL = {Geometric and Functional Analysis},
    VOLUME = {19},
      YEAR = {2009},
    NUMBER = {2},
     PAGES = {429--456},
}

@Article{zbMATH05003392,
 Author = {Kleinbock, Dmitry and Weiss, Barak},
 Title = {Badly approximable vectors on fractals},
 FJournal = {Israel Journal of Mathematics},
 Journal = {Israel J. Math.},
 Volume = {149},
 Pages = {137--170},
 Year = {2005},
}

@misc{yavicoli2025numbersomittingdigitscertain,
    title={Numbers omitting digits in certain base expansions}, 
    author={Alexia Yavicoli and Han Yu},
    year={2025},
    eprint={2503.09528},
    archivePrefix={arXiv},
    note={Available at \href{https://arxiv.org/abs/2503.09528}{arXiv:2503.09528}}
}

@Article{MissingAllrectanglesKeleti,
 Author = {Keleti, Tam{\'a}s},
 Title = {A {{\(1\)}}-dimensional subset of the reals that intersects each of its translates in at most a single point},
 FJournal = {Real Analysis Exchange},
 Journal = {Real Anal. Exchange},
 Volume = {24},
 Number = {2},
 Pages = {843--844},
 Year = {1999},
}

@Article{MissingcountablymeanytrianglessKeleti,
 Author = {Keleti, Tam{\'a}s},
 Title = {Construction of one-dimensional subsets of the reals not containing similar copies of given patterns},
 FJournal = {Analysis \& PDE},
 Journal = {Anal. PDE},
 Volume = {1},
 Number = {1},
 Pages = {29--33},
 Year = {2008},
}

@Article{Smallsetscontaining,
 Author = {Molter, Ursula and Yavicoli, Alexia},
 Title = {Small sets containing any pattern},
 FJournal = {Mathematical Proceedings of the Cambridge Philosophical Society},
 Journal = {Math. Proc. Cambridge Philos. Soc.},
 Volume = {168},
 Number = {1},
 Pages = {57--73},
 Year = {2020},
}

@Article{pattsparcesets,
 Author = {Chan, Vincent and {\L}aba, Izabella and Pramanik, Malabika},
 Title = {Finite configurations in sparse sets},
 FJournal = {Journal d'Analyse Math{\'e}matique},
 Journal = {J. Anal. Math.},
 Volume = {128},
 Pages = {289--335},
 Year = {2016},
}

@Article{HutchinsonResult,
 Author = {Hutchinson, John E.},
 Title = {Fractals and self similarity},
 FJournal = {Indiana University Mathematics Journal},
 Journal = {Indiana Univ. Math. J.},
 Volume = {30},
 Pages = {713--747},
 Year = {1981},
}

@misc{baker2025finitelybaseqexpansions,
    title={On finitely many base $q$ expansions}, 
    author={Simon Baker and George Bender},
    year={2025},
    eprint={2501.09582},
    archivePrefix={arXiv},
    url={https://arxiv.org/abs/2501.09582}, 
    note={Available at \href{https://arxiv.org/abs/2501.09582}{arXiv:2501.09582}}
}

%\newpage

\appendix

\section{Proofs of Theorem~\ref{thm:winning-set-dimension} auxiliary results}\label{Appendix}

Here we provide the proofs of the technical auxiliary results used in the proof of \Cref{thm:winning-set-dimension}.
We split the appendix into two parts.
\Cref{appendixA1} contains the proof of the dimensional estimate for the set $F$ constructed in the proof of \Cref{thm:winning-set-dimension} (see \Cref{Hausdorff}) and \Cref{appendixA2} is devoted to the proof of \Cref{claim:claim_B}. 
We utilise the definitions and notation from \Cref{sec:winng-sets-HD} throughout.

\subsection{Strongly tree-like sets and dimensional estimate for the set $F$}\label{appendixA1}

    In this section, we prove the lower bound on the Hausdorff dimension of the set $F$ constructed in the proof of \Cref{thm:winning-set-dimension}.
    This is the content of \Cref{Hausdorff}.
    To do this, we use a result of \cite[Section 2.4]{MR2581371} on a class of sets that includes the set $F$.

    \begin{defin}\label{defn:stl}
        A countable family $\mathcal{A}$ of compact subsets of $\mathbb{R}^n$ with positive Lebesgue measure is called \emph{strongly tree-like} with respect to the Lebesgue measure if $\mathcal{A}$ is the union of finite sub-collections $\mathcal{A}_k$, where $\mathcal{A}_0=\{A_0\}$ and the following are satisfied:
        \begin{itemize}
            \item[(TL1)] For all $k\in \N$ and $A,B\in \mathcal{A}_k$ with $A \neq B$, $\mathcal{L}^n(A\cap B)=\emptyset$.
            \item[(TL2)] For all $k\in\N$ and $B\in \mathcal{A}_k$, there exists $A\in \mathcal{A}_{k-1}$ such that $B\subseteq A$.
            \item[(TL3)] For all $k\in\N$ and $A\in \mathcal{A}_k$, there exists $B\in \mathcal{A}_{k+1}$ such that $B\subseteq A$.
            \item[(STL)] $\lim_{k\rightarrow\infty}\max_{A\in \mathcal{A}_k}\{\operatorname{diam}(A)\}=0$.
        \end{itemize}
    The non-empty limit set of $\mathcal{A}$, referred to as a \emph{strongly tree-like set}, is defined as
    \begin{align*}
    A_{\infty}=\bigcap_{k\in\N}\bigcup_{A\in \mathcal{A}_k}A.
    \end{align*}
    \end{defin}
    
    In the paper \cite{MR2581371}, strongly tree-like sets are defined for a larger class of measures.
    However, the above formulation with respect to Lebesgue measure is sufficient for our purposes.
    In \cite[Lemma~2.8]{MR2581371}, a lower bound on the Hausdorff dimension was obtained.
    In our setting, this takes the following form.

    \begin{lem}\label{lem:stl-Hausdorff-bound}
        Letting $A_{\infty}$ be a strongly tree-like set, and following the notation of \Cref{defn:stl}, the following holds:
        \begin{align}\label{eq: restatedDimEstimate}
            \dim_H(A_\infty)\geq n-\limsup_{k\rightarrow \infty} \frac{\sum\limits_{i=0}^k\log\left(\min\limits_{B\in\mathcal{A}_i}\left\{ \mathcal{L}^n\left(\bigcup\limits_{A\in\mathcal{A}_{i+1} }A\cap B\right) \left( \mathcal{L}^n(B) \right)^{-1} \right\}\right)}{\log(\max_{A\in\mathcal{A}_k}\{\operatorname{diam}(A)\})}.
        \end{align}
    \end{lem}
    Using this result, we obtain the dimension bound used in the proof of \Cref{thm:winning-set-dimension}.
    
    \begin{lem}\label{Hausdorff}
    Assuming the setting of \Cref{thm:winning-set-dimension}, and
    letting $K_{1}$ and $F$ be as in \eqref{eq:main_result_constant_K} 
    and \eqref{eq:sub_cantor_set} respectively,
        \begin{align}\label{eq:Hausdorf_Lemma}
        \dim_{H}(F)\geq   n-K_1\frac{\alpha}{\left\lvert\log(\max\{\beta_{jj}\})\right\rvert}.
        \end{align}
    \end{lem}

    \begin{proof}
    We observe that the set $F$ is a strongly tree-like set with respect to the Lebesgue measure.
    This can be seen by noting that each set in our tree construction is of positive Lebesgue measure and for each $\mathcal{B}_k$, two distinct elements of $\mathcal{B}_k$ can intersect only along their boundaries (satisfying (TL1)).
    The conditions of (TL3) and (STL) are additionally satisfied, by construction.

    To verify (TL2), we observe that $\mathcal{B}_{k+1}\subset \bigcup_{T\in \mathcal{B}_{k}}D_{(k+1)N+1}(T)$. 
    Therefore, for any element $T'\in \mathcal{B}_{k+1}$, there exists $T\in \mathcal{B}_k$ such that $T'\in D_{(k+1)N+1}(T)$. 
    The condition (TL2) follows by the definition of the collection $D_{(k+1)N+1}(T)$.
    
    Letting $K$ be as in \eqref{eq:the_constant_K}, applying \Cref{lem:stl-Hausdorff-bound} gives the following:
        \begin{align}\label{eq:orginal_haus_bound_messy}
        \text{dim}_H(F)&\geq n-\limsup\limits_{k\rightarrow \infty}\frac{(k+1)\log{(K\prod_{j=1}^n\beta_{jj}^N)}}{\log{\big( \max\{\beta_{jj}\}^{kN+1}2\rho \big)}}.
        \end{align}
    Observe that
        \begin{align*}
        &n-\limsup\limits_{k\rightarrow \infty}\frac{(k+1)\log{(K\prod_{j=1}^n\beta_{jj}^N)}}{\log{\big( \max\{\beta_{jj}\}^{kN+1}2\rho \big)}}=n-\limsup\limits_{k\rightarrow \infty}\frac{(k+1)\log{(K\prod_{j=1}^n\beta_{jj}^N)}}{(kN+1)\log{\big( \max\{\beta_{jj}\}}\big)}\\
        &=n-\frac{\log{(K\prod_{j=1}^n\beta_{jj}^N)}}{N\log{\big( \max\{\beta_{jj}\} \big)}} = n-\frac{\log(\prod_{j=1}^n\beta_{jj})}{\log\left( \max\{\beta_{jj}\} \right)}-\frac{\log\left(K\right)}{N\log(\max\{\beta_{jj}\})}\\
        &\geq n-\frac{\log(\prod_{j=1}^n\beta_{jj})}{\log( \max\{\beta_{jj}\} )}-\frac{\log((\prod_{j=1}^n\beta_{jj}^{-N})( \prod_{j=1}^n( 3^{-1} (1-5\beta_{jj}^N))-8^n (1+2^{2n+1})\delta))}{N\log(\max\{\beta_{jj}\})}\\
        &= n-\frac{\log(\prod_{j=1}^n\beta_{jj})}{\log( \max\{\beta_{jj}\})}-\frac{\log(( \prod_{j=1}^n( 3^{-1} (1-5\beta_{jj}^N))-8^n (1+2^{2n+1})\delta))}{N\log(\max\{\beta_{jj}\})}+\frac{\log(\prod_{j=1}^n\beta_{jj})}{\log(\max\{\beta_{jj}\})}.
        \end{align*}
        Since both $\max\{\beta_{jj}\}$ and $\prod_{j=1}^n( 3^{-1}(1-5\beta_{jj}^N))-8^n (1+2^{2n+1})\delta$ are positive and less than one, we conclude
        \begin{align*}
             \text{dim}_H(F) &\geq n-\left|\frac{\log( \prod_{j=1}^n ( 3^{-1}(1-5\beta_{jj}^N))-8^n (1+2^{2n+1})\delta)}{N\log(\max\{\beta_{jj}\})}\right|\\
            &\geq n-\frac{2\alpha\left|\log( \prod_{j=1}^n( 3^{-1} (1-5\beta_{jj}^N))-8^n (1+2^{2n+1})\delta)\right|}{\delta\left|\log(\max\{\beta_{jj}\})\right|}
            =n-K_1\frac{\alpha}{\left|\log(\max\{\beta_{jj}\})\right|}. \qedhere
        \end{align*}
    \end{proof}

We observe that the bound in \eqref{eq:Hausdorf_Lemma} is more useful than \eqref{eq:orginal_haus_bound_messy} in our computations.  Moreover, the difference between the bound in \eqref{eq:Hausdorf_Lemma} and \eqref{eq:orginal_haus_bound_messy} is small for small values of $\alpha$ or when $\prod_{j=1}^n\beta_{jj}$ is small.

\subsection{Proof of \Cref{claim:claim_B}}\label{appendixA2}

The aim of this section is to present a proof of \Cref{claim:claim_B} from the proof of \Cref{thm:winning-set-dimension}, which for convenience we restate below.

\setcounter{claimA}{1}

\begin{claimA}
        \textit{For al}l $k\in\N_0$ \textit{and} $T\in\mathcal{D}_{kN+1}'$, 
        \begin{align}\label{eq:ClaimB}
            \#\left( \mathcal{D}_{(k+1)N+1}(T)\cap \mathcal{D}_{(k+1)N+1}' \right)
            \geq \left(\prod_{j=1}^n\beta_{jj}^{-N}\right)\left( 3^{-n}\prod_{j=1}^{n} (1-5\beta_{jj}^N) -8^n (1+2^{2n+1})\delta\right) > 0.
        \end{align}
\end{claimA}

To aid in the readability of the proof of \Cref{claim:claim_B}, we provide a summary of the main steps in the proof.
%where we recall that we utilise the definitions and notation of \Cref{sec:winng-sets-HD}.
\begin{enumerate}
\item We use a counting argument to obtain a lower bound on the size of $\mathcal{D}_{(k+1)N+1}(T)$.
This is achieved by defining a set $Q$ whose points define a subset of $\mathcal{D}_{(k+1)N+1}(T)$.
That this set defines a subset of $\mathcal{D}_{(k+1)N+1}(T)$ is proved in \Cref{claim:claim_1}.
Following this, we provide a lower bound for the size of this subset and hence a lower bound on the size of the set $\mathcal{D}_{(k+1)N+1}(T)$.
This is the content of \Cref{claim:claim_2}.
\item We prove an upper bound on $\#(\mathcal{D}_{(k+1)N+1}(T)\setminus\mathcal{D}_{(k+1)N+1}')$ of the form
    \begin{align}\label{eq:summaryproofofClaimB_eq1}
    \#(\mathcal{D}_{(k+1)N+1}(T)\setminus\mathcal{D}_{(k+1)N+1}')\leq \sum_{t=1}^{(k+1)N} S_t,
    \end{align}
where heuristically $S_t$ tells how many possible plays of Player I are affected by Player II's deletion from the $t$-th level.
\item We split the sum in the right-hand-side of \eqref{eq:summaryproofofClaimB_eq1} into two sums, one considering deletion sets which have occurred early in the game, and the other considering more recent plays. We then obtain upper bounds for each of these sums. To do this we utilise \Cref{Claim (A)}, \Cref{claim:claim_3} and \Cref{claim:claim_4} stated below.
%are then proved and used to obtain an upper bound for the two sums.
%
\item The proof concludes by taking the difference of $\#\mathcal{D}_{(k+1)N+1}(T)$ and $\#(\mathcal{D}_{(k+1)N+1}(T)\setminus\mathcal{D}_{(k+1)N+1}')$, appropriately bounding and rearranging into the form given in the statement of \Cref{claim:claim_B}.
\end{enumerate}

Before presenting the proof of \Cref{claim:claim_B} we state and prove \Cref{Claim (A)}.
   
\begin{lem}\label{Claim (A)}
    Let $k\in\N$ be such that $k \not\equiv 1 \operatorname{mod} N$ and let $T\in \mathcal{E}_{k+1}$.
    If $z$ is the element of $E_k$ such that $\pi_{k}(T)=A^{k}(B[0,\rho])+z$, then $T\subseteq 2^{-1} A^{k}(B[0,\rho])+z$.
\end{lem}    
    \begin{proof}
        For each $T\in \mathcal{E}_{k+1}$, we have that 
        \begin{align}\label{eq:defn_of_r}
        T=A^{k+1}(B[0,\rho])+(2^{-1}\rho\beta_{11}^{k+1}r_{1},\dots,2^{-1} \rho\beta_{nn}^{k+1}r_{n})
        \end{align}
        for some $(r_{1},\dots,r_{n})\in \Z^n$.
        Moreover, $\pi_k(T)=A^{k}(B[0,\rho])+(2^{-1} \rho\beta_{11}^{k}w_{1},\dots,2^{-1}\rho\beta_{nn}^{k}w_{n})$ for some $(w_{1},\dots,w_{n})\in \Z^n$, with $w_{j}$ chosen so that $\lvert w_{j} - \beta_{jj} r_{j} \rvert$ is minimised for all $j \in \{1,\ldots,n\}$.
        This in tandem with the definition of $\pi_{k}$ and the sets $\mathcal{E}_k$ and $\mathcal{E}_k+1$, yields
        $\lvert 2^{-1} \rho \beta_{jj}^{k}w_{j}-2^{-1}\rho \beta_{jj}^{k+1}r_{j} \rvert \leq 4^{-1} \rho \beta_{jj}^{k}$.
        Thus, for $x = (x_{1},\ldots,x_{n}) \in T$ and $j \in \{1,\ldots,n\}$,
        \[
            \left\lvert x_{j}-\frac{\rho}{2}\beta_{jj}^{k}w_{j}\right\rvert \leq \left\lvert x_{j}-\frac{\rho}{2}\beta_{jj}^{k+1}r_{j}\right\rvert + \left\lvert \frac{\rho}{2}\beta_{jj}^{k}w_{j}-\frac{\rho}{2}\beta_{jj}^{k+1}r_{j} \right\rvert \leq \rho\beta_{jj}^{k+1} + \frac{\rho}{4}\beta_{jj}^{k}\leq 
            \frac{\rho}{2}\beta_{jj}^{k}.\qedhere
        \]
    \end{proof}

We now present the proof of \Cref{claim:claim_B}.

    \begin{proof}[Proof of \Cref{claim:claim_B}]
    Let $k\in\N_0$, let $T=A^{kN+1}(B[0,\rho])+z\in\mathcal{D}_{kN+1}'$ and let $r = (r_1,\ldots,r_n) \in \Z^n$ be such that $z=(3\rho \beta_{11}^{kN+1}r_1,...,3\rho \beta_{nn}^{kN+1}r_n)$.
    We first find a lower bound for $\#(\mathcal{D}_{(k+1)N+1}(T))$ and then find an upper bound for $\#(\mathcal{D}_{(k+1)N+1}(T)\setminus\mathcal{D}_{(k+1)N+1}')$.

    To obtain a lower bound for $\#\mathcal{D}_{(k+1)N+1}(T)$, we compute the cardinality of a set $Q\subseteq\R^n$, dependent on $T$, and show that for each point $x\in Q$ we have that $A^{(k+1)N+1}(B[0,\rho])+x$ is contained in $\mathcal{D}_{(k+1)N+1}(T)$.
    To this end, for each $j \in \{1,\ldots,n\}$, let $\gamma_j = \beta_{jj}^{-N}/6 - 1/3$, let $H \colon \R \rightarrow \Z$ be the function that maps each real number to its closest integer, rounding down when the fractional part is $1/2$, and define
    \begin{align*}
        Q= \{ 
            (x_{1}, \dots, x_{n})
            \in \R^n : \; & x_j = 3 \rho \beta_{jj}^{(k+1)N+1} ( H [ \beta_{jj}^{-N} r_n ] + l_j )\\
            &\text{for some} \; l_j \in \Z \; \text{satisfying} \; |\beta_{jj}^{-N}r_j-H[\beta_{jj}^{-N}r_j]-l_j|\leq \gamma_j \}.
    \end{align*}
    
    \begin{claim}\label{claim:claim_1}
    $\{A^{(k+1)N+1}(B[0,\rho])+x : x \in Q\}\subseteq \mathcal{D}_{(k+1)N+1}(T)$.
    \end{claim}

    \begin{proof}[Proof of \Cref{claim:claim_1}]
        Observe, by construction, that $Q \subseteq D_{(K+1)N + }$. Let $x=(x_1,...,x_n) \in Q$ be given.
        By definition of $Q$, for $j \in \{1,\ldots,n\}$, there exists $l_j \in \Z$ with $x_j=3\rho \beta_{jj}^{(k+1)N+1}\left(H\left[ \beta_{jj}^{-N} r_i \right]+l_j\right)$ and so
        \begin{align*}
            \left\lvert z_j-x_j \right\rvert
            &=3\rho \beta_{jj}^{(k+1)N+1} \left\lvert \beta_{jj}^{-N} r_j-H[\beta_{jj}^{-N} r_j]-l_j \right\rvert\\
            &\leq 3\rho \beta_{jj}^{(k+1)N+1}\left(\frac{1}{6}\beta_{jj}^{-N}-\frac{1}{3}\right)=\frac{1}{2}\beta_{jj}^{kN+1}\rho- \beta_{jj}^{(k+1)N+1}\rho.
        \end{align*}
        Therefore, $\lvert z_j - (x_j+v_j) \rvert \leq \lvert v_j \rvert +\lvert z_j-x_j \rvert \leq\frac{1}{2}\beta_{jj}^{kN+1}\rho$, for all $v = (v_{1}, \dots, v_{n}) \in A^{(k+1)N+1} (B[0,\rho])$ and $j \in \{1,\ldots,n\}$.
        Hence, $x+v \in 2^{-1} A^{kN+1} (B[0,\rho])+z$ and so 
        \[
        A^{(k+1)N+1} (B[0,\rho]) + x\subseteq 2^{-1} A^{kN+1} (B[0,\rho]) + z.\qedhere
        \]
    \end{proof}

    \begin{claim}\label{claim:claim_2}
    $\displaystyle \# \mathcal{D}_{(k+1)N+1} (T) \geq \# Q = \prod_{j=1}^{n} \left\lfloor 2\gamma_j \right\rfloor$.
    \end{claim}
    
    \begin{proof}[Proof of \Cref{claim:claim_2}]
        The first inequality follows from \Cref{claim:claim_1}, and so we focus on the second inequality.
        The cardinality of $Q$ equals the number of vectors $l = (l_1,\ldots,l_n) \in \Z^{n}$ with $\lvert \beta_{jj}^{-N} r_j - H[\beta_{jj}^{-N} r_j] - l_j \rvert \leq \gamma_j$ for all $j \in \{1,\ldots,n\}$.
        If $l_j \leq \beta_{jj}^{N} r_j - H[\beta_{jj}^{-N} r_j]$, then we require $l_j \geq \beta_{jj}^{-N} r_j - H[\beta_{jj}^{-N}] - \gamma_j$, which implies $l_j \in \left[ \beta_{jj}^{-N} r_j - H [\beta_{jj}^{-N} r_j] - \gamma_j, \, \beta_{jj}^{-N} r_j - H [\beta_{jj}^{-N} r_j] \right]$.
        Similarly, if $l_j > \beta_{jj}^{N} r_j - H[\beta_{jj}^{-N} r_j]$, then we require $l_j\in \left(\beta_{jj}^{-N} r_j - H [\beta_{jj}^{-N} r_j], \, \beta_{jj}^{-N} r_j - H [\beta_{jj}^{-N} r_j] + \gamma_j \right]$.
        Thus, the permissible values of $l_j$ are precisely those in the set $[ \beta_{jj}^{-N}z_j-H[\beta_{jj}^{-N}z_j] - \gamma_j, \beta_{jj}^{-N}z_j-H[\beta_{jj}^{-N}z_j] + \gamma_j ]\cap \Z$, the cardinality of which is bounded below by $\lfloor 2\gamma_j \rfloor$.
        Hence, we conclude that $\# Q \geq\prod_{j=1}^{n}  \lfloor 2\gamma_j \rfloor$.
    \end{proof}
        
 We now pursue an upper bound for $\#(\mathcal{D}_{(k+1)N+1}(T)\setminus\mathcal{D}_{(k+1)N+1}')$.
    Note that $\mathcal{D}_{(k+1)N+1}(T)\setminus\mathcal{D}_{(k+1)N+1}'$ is precisely the set of $T' \in \mathcal{D}_{(k+1)N+1}(T)$ such that $\phi_{(k+1)N+1} (T') > (\delta\prod_{j=1}^n\beta_{jj}^{(k+1)N+1})^c$.
    Thus,
    \begin{align*}
        \#\left( \mathcal{D}_{(k+1)N+1}(T)\setminus\mathcal{D}_{(k+1)N+1}' \right) \leq \sum_{T' \in \mathcal{D}_{(k+1)N+1} (T)} \min \left\{ 1, \frac{\phi_{(k+1)N+1} (T')}{\left(\delta\prod_{j=1}^n\beta_{jj}^{(k+1)N+1}\right)^c} \right\} \leq \sum_{t=1}^{(k+1)N} S_t,
    \end{align*}
    where, for each $t \in \{1,\ldots,(k+1)N\}$,
    \begin{align*}
        S_t = \sum_{T' \in \mathcal{D}_{(k+1)N+1} (T)} \,\sum_{(q_{i,t},y_{i,t}) \in \mathcal{A}_t^{*} (T')} \min \left\{ 1, \, \left( \prod_{j=1}^{n} \beta_{jj}^{q_{i,t}} \right)^c \left(\delta\prod_{j=1}^n\beta_{jj}^{(k+1)N+1}\right)^{-c} \right\}.
    \end{align*}
    We will bound separately the sums $\sum_{t=1}^{kN} S_t$ and $\sum_{t=kN+1}^{(k+1)N} S_t$.
    We first consider $\sum_{t=1}^{kN} S_t$.
    For all $T' \in \mathcal{D}_{(k+1)N+1} (T)$, we have that $T' \subseteq 2^{-1} A^{kN+1} (B[0,\rho]) + z$.
    Thus, by \Cref{claim:claim_A}, we have that $\pi_{(k+1)N+1,kN+1} (T') = T$, and so, for $t \in \{1, \dots, kN \}$,
        \begin{align*}
        \mathcal{A} (\pi_{(k+1)N+1,t} (T'),\ldots,\pi_{(k+1)N+1,1} (T')) \subseteq \mathcal{A} (\pi_{kN+1,t} (T),\ldots,\pi_{(kN+1,1} (T)).
        \end{align*}
    Hence, for all $t \leq kN$, we obtain the following upper bound for $S_t$:
    \begin{align}
    \begin{split}\label{EQ:small-St-bound}
        S_t
        &\leq \sum\limits_{(q_{i,t},y_{i,t})\in \mathcal{A}^*_t(T)} \,\,\sum_{\substack{T'\in\mathcal{D}_{(k+1)N+1}(T): \\ T'\cap (A^{q_{i,t}}(B[0,\rho])+y_{i,t})\neq\emptyset}}\min \left\{ 1, \, \left( \prod_{j=1}^{n} \beta_{jj}^{q_{i,t}} \right)^c \left(\delta\prod_{j=1}^n\beta_{jj}^{(k+1)N+1}\right)^{-c} \right\}\\
        &\leq \sum\limits_{(q_{i,t},y_{i,t})\in \mathcal{A}^*_t(T)} n (T,q_{i,t},y_{i,t}) \, \min \left\{ 1, \left(\prod\limits_{j=1}^n \beta_{jj}^{q_{i,t}}\right)^c \left(\delta\prod\limits_{j=1}^n\beta_{jj}^{(k+1)N+1}\right)^{-c} \right\},
    \end{split}
    \end{align}
    where
    \begin{align}\label{EQ:n-def}
        n (T,q_{i,t},y_{i,t}) = \# \left\{ T' \in \mathcal D_{(k+1)N+1} (T) : T' \cap (A^{q_{i,t}} (B[0,\rho]) + y_{i,t}) \neq \varnothing \right\}.
    \end{align}

    We now consider the case $kN+1\leq t< (k+1)N+1$.
    If $T' \in \mathcal{D}_{(k+1)N+1} (T)$ and $z$ is the element in $E_k$ such that $\pi_{(k+1)N+1,t} (T') = A^{t} (B[0,\rho]) + z$, then $T' \subseteq 2^{-1} A^{t} (B[0,\rho]) + z$ by \Cref{Claim (A)}.
    Therefore, for all $T'\in \mathcal{D}_{(k+1)N+1}(T)$, we have that $T'\in \mathcal{D}_{(k+1)N+1}(\pi_{(k+1)N+1,t} (T'))$, and so
    \begin{align*}
    \mathcal{D}_{(k+1)N+1}(T)&\subseteq \bigcup_{T'\in\mathcal{D}_{(k+1)N+1}(T)} \mathcal{D}_{(k+1)N+1}(\pi_{(k+1)N+1,t}(T'))\subseteq \bigcup_{\substack{T''\in\mathcal{E}_t:  T''\subseteq T}} \mathcal{D}_{(k+1)N+1}(T'').
    \end{align*}
    Since, for $T' \in \mathcal{D}_{(k+1)N +1}(T)$, we have $\pi_{(k+1)N+1,t}(T')=T''$ for some $T''\in \mathcal{E}_t$ with $T''\subseteq T$, we obtain the following upper bound for $S_t$:
    \begin{align}
    \begin{split}\label{EQ:big-St-bound}
        S_t &\leq \sum_{\substack{T''\in\mathcal{E}_t\\ 
        T''\subseteq T}}\,\,\sum_{T'\in\mathcal{D}_{(k+1)N+1}(T'')} \,\,\sum\limits_{\substack{(q_{i,t},y_{i,t})\in \mathcal{A}^*_t(T''): \\ (A^{q_{i,t}}(B[0,\rho])+y_{i,t})\cap T'\neq \emptyset}}\min \left\{ 1, \, \left(\prod\limits_{j=1}^n \beta_{jj}^{q_{i,t}}\right)^c \left(\delta\prod\limits_{j=1}^n\beta_{jj}^{(k+1)N+1}\right)^{-c}  \right\}\\
        &\leq \sum_{\substack{T''\in\mathcal{E}_t\\ 
        T''\subseteq T}}\,\,\sum_{(q_{i,t},y_{i,t})\in \mathcal{A}^*_t(T'')} \,\,\sum\limits_{\substack{T'\in\mathcal{D}_{(k+1)N+1}(T''): \\ (A^{q_{i,t}}(B[0,\rho])+y_{i,t})\cap T'\neq \emptyset}}\min \left\{ 1, \, \left(\prod\limits_{j=1}^n \beta_{jj}^{q_{i,t}}\right)^c \left(\delta\prod\limits_{j=1}^n\beta_{jj}^{(k+1)N+1}\right)^{-c}  \right\}\\
        &\leq \sum_{\substack{T''\in\mathcal{E}_t: 
        T''\subseteq T}}\,\,\sum_{\substack{(q_{i,t},y_{i,t}) \in \mathcal{A}^*_t(T'')}} n(T'',q_{i,t},y_{i,t}) \, \min \left\{ 1, \left(\prod\limits_{j=1}^n \beta_{jj}^{q_{i,t}}\right)^c \left(\delta\prod\limits_{j=1}^n\beta_{jj}^{(k+1)N+1}\right)^{-c} \right\},
    \end{split}
    \end{align}
    where $n (T'',q_{i,t},y_{i,t})$ is as defined in \eqref{EQ:n-def}.
    It remains to find appropriate upper bounds for $n(T,q_{i,t},y_{i,t})$ when $t<kN+1$, and $n (T'',q_{i,t},y_{i,t})$ when $t \in \{kN+1,\dots,(k+1)N\}$ and $T''\in \mathcal{E}_t$.
    We combine these two bounds into the following claim.

    \begin{claim}\label{claim:claim_3}
    For $t \in \{kN+1,\ldots,(k+1)N\}$, $m \in \{1, \dots, t\}$, $T'' \in \mathcal{E}_t$ and $(q_{i,m},y_{i,m}) \in \mathcal{A}_{m}^{*}(T'')$,
    \begin{align*}
        n (T'',q_{i,m},y_{i,m}) \leq 4^n\left(\prod_{j=1}^n\left( \beta_{jj}^{q_{i,m}}+ \beta_{jj}^{(k+1)N+1} \right)\right)\left(\prod\limits_{j=1}^n \beta_{jj}^{(k+1)N+1}\right)^{-1}.
    \end{align*}
    \end{claim}
    
    \begin{proof}[Proof of \Cref{claim:claim_3}]
        For ease of notation, we write 
        \begin{align*}
        E = \left\{ T'\in\mathcal{D}_{(k+1)N+1}(T''):T'\cap(A^{q_{i,m}}(B[0,\rho])+y_{i,m})\neq\emptyset   \right\}
        \end{align*}
        for the purposes of this proof, so that $n (T'',q_{i,m},y_{i,m}) = \# E$.
        Since elements of $\mathcal{D}_{(k+1)N+1}$ are pairwise disjoint and $\mathcal{L}^n (T') = \prod_{j=1}^{n} 2 \rho \beta_{jj}^{(k+1)N+1}$ for all $T' \in \mathcal{D}_{(k+1)N+1}$, 
        \begin{align*}
            (\# E) \left( \prod_{j=1}^{n} 2 \rho \beta_{jj}^{(k+1)N+1} \right) = \mathcal{L}^n \left( \bigcup_{U \in E} U \right),
        \end{align*}
        where $\mathcal{L}^n$ denotes the $n$-dimensional Lebesgue measure.
        Moreover, letting $y^{(j)}_{i,m}$ denote the $j$-th component of $y_{i,m}$, for all $T' \in E$,
            \begin{align*}
            T' \subseteq  \prod_{j=1}^n\left[y_{i,m}^{(j)}-(2\rho\beta_{jj}^{q_{i,m}}+2\rho\beta_{jj}^{(k+1)N+1}),\,y^{(j)}_{i,m}+(2\rho\beta_{jj}^{q_{i,m}}+2\rho\beta_{jj}^{(k+1)N+1})\right].
            \end{align*}
        Hence,
        \begin{align*}
            (\# E) \left( \prod_{j=1}^{n} 2 \rho \beta_{jj}^{(k+1)N+1} \right) &\leq \mathcal{L}^n\left( \prod_{j=1}^n\left[y_{i,m}^{(j)}-2\rho(\beta_{jj}^{q_{i,m}}+\beta_{jj}^{(k+1)N+1}),\,y^{(j)}_{i,m}+2\rho(\beta_{jj}^{q_{i,m}}+\beta_{jj}^{(k+1)N+1})\right]\right)\\
            &\leq 4^n\left(\prod_{j=1}^n\left( \beta_{jj}^{q_{i,m}}+ \beta_{jj}^{(k+1)N+1} \right)\right). \qedhere
        \end{align*}
    \end{proof}
    
    \begin{claim}\label{claim:claim_4}
    Given $t\in[kN+1,(k+1)N]\cap\N_0$ and $T''\in \mathcal{E}_t$ we have
        \begin{enumerate}[label=(\alph*)]
        \item \quad
        $\displaystyle
        \sum_{(q_{i,t},y_{i,t})\in \mathcal{A}^*_t(T'')} n(T'',q_{i,t},y_{i,t}) \,\min \left\{ 1,\left(\prod\limits_{j=1}^n \beta_{jj}^{q_{i,t}}\right)^c \,\left(\delta\prod\limits_{j=1}^n\beta_{jj}^{(k+1)N+1}\right)^{-c} \right\}$\\[0.75em]
        $\displaystyle \hphantom{\quad \sum_{(q_{i,t},y_{i,t})\in \mathcal{A}^*_t(T'')}} \leq 8^n\left(\prod\limits_{j=1}^n\beta_{jj}^{t-(k+1)N-1}\right)\max\left\{ \alpha,\left( \frac{\alpha}{\delta} \right)^c\left(\prod\limits_{j=1}^n\beta_{jj}^{(k+1)N+1-t}\right)^{1-c} \right\}$\\[0.75em]
        \item \quad
            $\displaystyle
            \sum\limits_{t=1}^{kN} \sum\limits_{(q_{i,t},y_{i,t})\in \mathcal{A}^*_t(T)} n(T,q_{i,t},y_{i,t}) \,\min \left\{ 1,\left(\prod\limits_{j=1}^n \beta_{jj}^{q_{i,t}}\right)^c \,\left(\delta\prod\limits_{j=1}^n\beta_{jj}^{(k+1)N+1}\right)^{-c} \right\}$\\[0.75em]
            $\displaystyle \hphantom{\quad \sum_{(q_{i,t},y_{i,t})\in \mathcal{A}^*_t(T'')}} \leq8^n\left(\prod\limits_{j=1}^n \beta_{jj}^{-N}\right)\left( \max\left\{ \delta, \left(\prod_{j=1}^n\beta_{jj}^{N}\right)^{1-c} \right\}\right).$
        \end{enumerate}
    \end{claim}
        
    \begin{proof}[Proof of \Cref{claim:claim_4}]
        By \Cref{claim:claim_3},
        \begin{align}
        &\sum_{(q_{i,t},y_{i,t})\in \mathcal{A}^*_t(T'')} \,\, n(T'',q_{i,t},y_{i,t}) \,\min \left\{ 1,\left(\prod\limits_{j=1}^n \beta_{jj}^{q_{i,t}}\right)^c \,\left(\delta\prod\limits_{j=1}^n\beta_{jj}^{(k+1)N+1}\right)^{-c} \right\}\nonumber\\
        &\leq 4^n\left(\prod\limits_{j=1}^n \beta_{jj}^{(k+1)N+1}\right)^{-1}\sum_{(q_{i,t},y_{i,t})} \,\,\min \left\{ 1, \frac{\left(\prod\limits_{j=1}^n \beta_{jj}^{q_{i,t}}\right)^c}{\left(\delta\prod\limits_{j=1}^n\beta_{jj}^{(k+1)N+1}\right)^c} \right\}\left(\prod_{j=1}^n\left( \beta_{jj}^{q_{i,t}}+2\rho \beta_{jj}^{(k+1) N+1} \right)\right)\!.\label{eq:first_inequality_of_proof_of_claim_4}
        \end{align}
        By the binomial theorem, there exist constants $a_{i,t}, b_{i,t} > 0$ with $a_{i,t}+b_{i,t} = 2^n$, such that 
        \begin{align}\label{EQ:sum-product-inequality}
        \prod\limits_{j=1}^n\left( \beta_{jj}^{q_{i,t}}+ \beta_{jj}^{(k+1)N+1} \right)\leq a_{i,t}\prod_{j=1}^n \beta_{jj}^{q_{i,t}}+b_{i,t}\prod_{j=1}^n2\beta_{jj}^{(k+1) N+1}.
        \end{align}
        Moreover, we note that for all $x,y,a,b,c,\gamma > 0$,
        \begin{align}\label{magicinequal}
            \min\left\{ 1,\frac{x^c}{(\gamma y)^c} \right\}(ax+by)\leq (a+b)x^c\max\left\{ x^{1-c},\frac{y^{1-c}}{\gamma^c} \right\}.
        \end{align}
        This can be verified by considering separately the cases $x \leq y$ and $y \leq x$.
        
        Applying \eqref{EQ:sum-product-inequality} and \eqref{magicinequal}, with $\gamma=\delta$, $x=x_{i,t}=\prod_{j=1}^n \beta_{jj}^{q_{i,t}},y=\prod_{j=1}^n \beta_{jj}^{(k+1)N+1}$, to \eqref{eq:first_inequality_of_proof_of_claim_4} we obtain
        \begin{align}
        &\hspace{-1em}\sum_{(q_{i,t},y_{i,t}) \in \mathcal{A}^*_t(T'')} \,\, n(T'',q_{i,t},y_{i,t}) \,\min \left\{ 1,\left(\prod\limits_{j=1}^n \beta_{jj}^{q_{i,t}}\right)^c \,\left(\delta\prod\limits_{j=1}^n\beta_{jj}^{(k+1)N+1}\right)^{-c} \right\}\nonumber\\[0.25em]
        &=\sum_{(q_{i,t},y_{i,t})\in \mathcal{A}^*_t(T'')} \,\, n(T'',q_{i,t},y_{i,t}) \,\min \{ 1, x_{i,t}^{c} (\delta y)^{-c} \}\nonumber\\[0.25em]
        &\leq 4^n y^{-1}\sum_{(q_{i,t},y_{i,t})\in \mathcal{A}^*_t(T'')} \,\,\min \left\{ 1, x_{i,t}^{c} (\delta y)^{-c} \right\}(a_{i,t}x_{i,t}+b_{i,t}y)\nonumber\\[0.25em]
        &\leq 8^n y^{-1}\sum_{(q_{i,t},y_{i,t})\in \mathcal{A}^*_t(T'')} x_{i,t}^c\max\{ x_{i,t}^{1-c},\delta^{-c}y^{1-c} \}\nonumber\\[0.25em]
        &= 8^n\left(\prod\limits_{j=1}^n \beta_{jj}^{(k+1)N+1}\!\right)^{\!\!-1}\!\!\!\!\!\!\sum_{(q_{i,t},y_{i,t})\in \mathcal{A}^*_t(T'')} \!\!\left(\prod\limits_{j=1}^n \beta_{jj}^{q_{i,t}}\!\right)^{\!\!c}\!\!\max\left\{ \left(\prod\limits_{j=1}^n \beta_{jj}^{q_{i,t}}\!\right)^{\!\!1-c}\hspace{-1.25em}, \hspace{0.75em}\delta^{-c} \left(\prod\limits_{j=1}^n \beta_{jj}^{(k+1)N+1}\!\right)^{\!\!1-c} \right\}\!.\label{eq:***}
        \end{align}
        Recall that by the rules of the game, we require
            \begin{align}\label{EQ:play-restriction-inequality}
            \sum\limits_{(q_{i,t},y_{i,t})} \,\,\left(\prod\limits_{j=1}^n \beta_{jj}^{q_{i,t}}\right)^c\leq \left( \alpha\prod_{j=1}^n\beta_{jj}^{t} \right)^c,
            \end{align}            
        which implies that for all $(q_{i,t},y_{i,t}) \in \mathcal{A}^*_t(T'')$,
            \begin{align}\label{eq:******}
            \prod\limits_{j=1}^n \beta_{jj}^{q_{i,t}} \leq\alpha\prod_{j=1}^n\beta_{jj}^{t}.
            \end{align}
        Combining \eqref{EQ:play-restriction-inequality}, and hence also \eqref{eq:******}, with \eqref{eq:***} yields (a).
        To obtain (b), we note, since $T\in \mathcal{D}_{kN+1}'$, that
        \begin{align}\label{eq:****}
            \sum\limits_{t=1}^{kN} \sum\limits_{(q_{i,t},y_{i,t})\in \mathcal{A}^*_t(T)}\left(\prod\limits_{j=1}^n \beta_{jj}^{q_{i,t}}\right)^c\leq \left(\delta\prod\limits_{j=1}^n  \beta_{jj}^{kN+1}\right)^c
        \end{align}    
        which implies that for all $t \in \{1, \dots, kN \}$ and  $(q_{i,t},y_{i,t}) \in \mathcal{A}^*_t(T)$,
        \begin{align}\label{eq:*****}
            \prod\limits_{j=1}^n \beta_{jj}^{q_{i,t}} \leq \delta\prod\limits_{j=1}^n \beta_{jj}^{kN+1}.
        \end{align}
     Using similar arguments, but with \eqref{eq:****}, and hence also \eqref{eq:*****}, one can obtain (b).
    \end{proof}
    
    Returning to the proof of Claim \ref{claim:claim_B}, we now obtain that 
    \begin{align*}
        &\#(\mathcal{D}_{(k+1)N+1}(T)\setminus\mathcal{D}_{(k+1)N+1}')\\
        &\leq \sum\limits_{t=1}^{kN} \sum\limits_{(q_{i,t},y_{i,t})\in \mathcal{A}^*_t(T)} n (T,q_{i,t},y_{i,t}) \, \min \left\{ 1, \left(\prod\limits_{j=1}^n \beta_{jj}^{q_{i,t}}\right)^c \left(\delta\prod\limits_{j=1}^n\beta_{jj}^{(k+1)N+1}\right)^{-c} \right\}\\
        &\hspace{1em} +\sum_{t=kN+1}^{(k+1)N}\, \sum_{\substack{T''\in\mathcal{E}_t\\ 
        T''\subseteq T}}\,\,\sum_{\substack{(q_{i,t},y_{i,t})\\ \in \mathcal{A}^*_t(T'')}} n(T'',q_{i,t},y_{i,t}) \, \min \left\{ 1, \left(\prod\limits_{j=1}^n \beta_{jj}^{q_{i,t}}\right)^c \left(\delta\prod\limits_{j=1}^n\beta_{jj}^{(k+1)N+1}\right)^{-c} \right\}\\
        &\leq 8^n \left(\prod\limits_{j=1}^n \beta_{jj}^{-N}\right)  \max\left\{ \delta, \left(\prod_{j=1}^n\beta_{jj}^{N}\right)^{1-c} \right\}\\
        &\hspace{1em} +8^n\sum\limits_{t=kN+1}^{(k+1)N}\#\{ T''\in\mathcal{E}_t:T''\subseteq T \} \left(\prod\limits_{j=1}^n\beta_{jj}^{t-(k+1)N-1}\right)\max\left\{ \alpha,\left( \frac{\alpha}{\delta} \right)^c\left(\prod\limits_{j=1}^n\beta_{jj}^{(k+1)N+1-t}\right)^{1-c} \right\},
    \end{align*}
    where the first inequality follows by combining \eqref{EQ:small-St-bound} and \eqref{EQ:big-St-bound}, and the second is a consequence of \Cref{claim:claim_4}.
    Observe, for all $t \in \{kN+1,\ldots,(k+1)N\}$, that the sets $\{T''\in\mathcal{E}_t:T''\subseteq T \}$ and
    \begin{align*}
        \bigg\{ A^{t}(B[0, \rho]) + \left( \frac{\rho}{2}\beta_{11}^{t}x_{1}, \dots, \frac{\rho}{2}\beta_{nn}^{t}x_{n} \right) : \;
        &\text{for all} \; j\in \{1,...,n\}, x_{j} \in \Z \; \text{and} \\
        &\left\lvert 3\rho\beta_{jj}^{kN+1}r_{j}-\frac{\rho}{2}\beta_{jj}^tx_{j} \right\rvert \leq \rho\beta_{jj}^{kN+1}-\rho\beta_{jj}^t \;  \bigg\},
    \end{align*}
    where $(r_{1}, \dots, r_{n})$ is as in \eqref{eq:defn_of_r}, are equal. Therefore, if $x \in \Z^{n}$ is such that 
        \begin{align*}
        A^{t}(B[0, \rho]) + \left( \frac{\rho}{2}\beta_{11}^{t}x_{1}, \dots, \frac{\rho}{2}\beta_{nn}^{t}x_{n} \right) \in \{T''\in\mathcal{E}_t:T''\subseteq T \},
        \end{align*}
    then for each $j \in \{1,...,n\}$ we have that 
    \begin{align*}
        x_{j} \in [6 r_{j} \beta_{jj}^{(kN+1)-t}+2-2\beta_{jj}^{(kN+1)-t},6r_{j}\beta_{jj}^{(kN+1)-t}-2+2\beta_{jj}^{(kN+1)-t}]\cap \Z.
    \end{align*}
    Therefore, the set $\{T''\in\mathcal{E}_t:T''\subseteq T \}$, contains at most $4\beta_{jj}^{kN+1-t}$ elements.
    Thus, we have the following upper bound for $\#(\mathcal{D}_{(k+1)N+1}(T)\setminus\mathcal{D}_{(k+1)N+1}')$:
    \begin{align*}
     &8^n\left(\prod\limits_{j=1}^n \beta_{jj}^{-N}\right)\left( \max\left\{ \delta, \left(\prod_{j=1}^n\beta_{jj}^{N}\right)^{1-c} \right\} +4^n\sum\limits_{t=kN+1}^{(k+1)N}\max\left\{ \alpha,\left( \frac{\alpha}{\delta} \right)^c\left(\prod\limits_{j=1}^n\beta_{jj}^{(k+1)N+1-t}\right)^{1-c} \right\}\right)\\
    &\leq 8^n\left(\prod\limits_{j=1}^n \beta_{jj}^{-N}\right)\left( \max\left\{ \delta, \left(\prod_{j=1}^n\beta_{jj}^{N}\right)^{1-c} \right\} +4^n\sum\limits_{t=kN+1}^{(k+1)N}\alpha+4^n\sum\limits_{t=kN+1}^{(k+1)N}\left( \frac{\alpha}{\delta} \right)^c\left(\prod\limits_{j=1}^n\beta_{jj}^{(k+1)N+1-t}\right)^{1-c} \right)\\
    &\leq 8^n\left(\prod\limits_{j=1}^n \beta_{jj}^{-N}\right)\left( \max\left\{ \delta, \left(\prod_{j=1}^n\beta_{jj}^{N}\right)^{1-c} \right\} +4^nN\alpha+4^n\sum\limits_{t=0}^{\infty}\left( \frac{\alpha}{\delta} \right)^c\left(\prod\limits_{j=1}^n\beta_{jj}^{t}\right)^{1-c} \right)
    \end{align*}
    It remains to bound each of the three terms in the final bracket.
    Recall that $N = \lfloor \delta \alpha^{-1} \rfloor$, so we have that $N \alpha \leq \delta$.
    We next show that $(\alpha \delta^{-1})^{c} \sum_{t=0}^{\infty} (\prod_{j=1}^n \beta_{jj})^{(1-c)t} \leq \delta$.
    To this end, recalling
    \begin{align*}
        \alpha^c \leq \delta^{2} \left( 1 - \left( \prod_{j=1}^{n} \beta_{jj} \right)^{1-c} \right),
    \end{align*}
    we have that
    \begin{align*}
        \left( \frac{\alpha}{\delta} \right)^c\sum\limits_{t=0}^{\infty}\left(\prod\limits_{j=1}^n\beta_{jj}\right)^{(1-c)t}
        = \left( \frac{\alpha}{\delta} \right)^c\left(1-\left(\prod\limits_{j=1}^n\beta_{jj}\right)^{(1-c)}\right)^{-1}
        \leq \delta^{2-c}\leq \delta.
    \end{align*}
    Finally, we show that $\prod_{j=1}^n \beta_{jj}^{N(1-c)} \leq \delta$.
    Since $\alpha^c \leq \delta^2$, we have that $(\delta \alpha^{-1})^c \geq\delta^{-(2-c)} \geq 1$, which implies $\delta \alpha^{-1} \geq 1$; thus, $N = \lfloor \delta \alpha^{-1} \rfloor \geq 2^{-1} \delta \alpha^{-1}$.
    Further, since $\alpha, c \in (0, 1)$ and since the function $x \mapsto \log (x^{-1}) + x - 1$ is positive for all $x \in (0,1]$,
\begin{align*}
        \alpha
        \leq \alpha^c\leq \delta^{2} \left(1-\left(\prod_{j=1}^n\beta_{jj}\right)^{1-c}\right)
        \leq \delta^{2} \left|\log\left(\left(\prod_{j=1}^n \beta_{jj} \right)^{1-c}\right)\right| 
        = \delta^{2}(1-c)\left|\log\left(\prod_{j=1}^n \beta_{jj} \right)\right|,
        \end{align*}
        This is tandem with the fact that $(2 \delta)^{-1} \geq - \log(\delta)$, yields
        \begin{align*}
            N(1-c)\left|\log\left(\prod_{j=1}^n \beta_{jj} \right)\right|\geq N\alpha \delta^{-2} \geq \frac{1}{2}\delta \delta^{-2} \geq |\log(\delta)|.
        \end{align*}
    Since $\delta<1$ and $\prod_{j=1}^n\beta_{jj}<1$, it follows that $\prod_{j=1}^{n} \beta_{jj}^{N(1-c)} \leq \delta$.
    Therefore, 
    \begin{align*}
        \#(\mathcal{D}_{(k+1)N+1}(T)\setminus\mathcal{D}_{(k+1)N+1}')\leq 8^n\left(\prod\limits_{j=1}^n \beta_{jj}^{-N}\right) (1+2^{2n+1})\delta
    \end{align*}
    and hence a lower bound for $\#( \mathcal{D}_{(k+1)N+1}(T)\cap \mathcal{D}_{(k+1)N+1}')$ is
    \[
            \max\left\{0, \left(\prod_{j=1}^n\beta_{jj}^{-N}\right)\left( \prod_{j=1}^n\left( \frac{1}{3}(1-5\beta_{jj}^N)\right)-8^n (1+2^{2n+1})\delta\right)\right\}.
            \qedhere
    \]
    \end{proof}
\end{document}